\documentclass{amsart} % Nicer than default article style: less
\usepackage{amsmath,amsthm}     % Handy math stuff, theorem environments.
\usepackage{amssymb,empheq}            % Fancy math symbols.
\usepackage{euscript}           % Nice script font.
\usepackage{graphicx,enumerate,calc,lscape,color}
\definecolor{red}{rgb}{1,0,0}
\usepackage[matrix,arrow,curve,frame]{xy}    % XY-pic diagram pac
\usepackage{hyperref}

\xymatrixcolsep{1.9pc}                          % Adjust size of diagrams.
\xymatrixrowsep{1.9pc}
\newdir{ >}{{}*!/-5pt/\dir{>}}                  % Make better tailed arrows

\def\longbfib{\DOTSB\twoheadleftarrow\joinrel\relbar}

\def\longfib{\DOTSB\relbar\joinrel\twoheadrightarrow}

\newtheorem{thm}[subsection]{Theorem}
\newtheorem{defn}[subsection]{Definition}
\newtheorem{prop}[subsection]{Proposition}

\newtheorem{cor}[subsection]{Corollary}
\newtheorem{lemma}[subsection]{Lemma}

\theoremstyle{definition}  % Bold headings and Roman body text.

\newtheorem{remark}[subsection]{Remark}

  {\end{list}}

\newcommand{\dfn}{\textbf} % Make defined words bold.

\newcommand{\mdfn}[1]{\dfn{\mathversion{bold}#1}} % Even make math bold

\newcommand{\Wedge}             {\vee}

\newcommand{\iso}               {\cong}
\DeclareMathOperator{\isoc}{iso}
\newcommand{\cat}{\EuScript}    % Use \EuScript to name a category.
\newcommand{\cA}{{\cat A}}      % Only seems to work for caps, and only gets
\newcommand{\cC}{{\cat C}}
\newcommand{\cD}{{\cat D}}

\newcommand{\cM}{{\cat M}}

\newcommand{\cQ}{{\cat Q}}

\newcommand{\cT}{{\cat T}}

\newcommand{\cY}{{\cat Y}}

\newcommand{\kG}{{\mathcal G}}

\newcommand{\Set}{{\cat Set}}

\newcommand{\sSet}{{s{\cat Set}}}
\newcommand{\Kan}{{\cat Kan}}

\newcommand{\Cat}{{\cat Cat}}

\newcommand{\ho}{\text{Ho}\,}

\newcommand{\sSetJ}{\sSet_J}
\newcommand{\field}[1]  {\mathbb #1} % Use blackboard bold for these sets

\newcommand{\N}         {\field N}

\newcommand{\C}         {\field C}

\DeclareMathOperator*{\colim}{colim}

\DeclareMathOperator*{\hocolim}{hocolim}

\DeclareMathOperator{\Hom}{Hom}

\DeclareMathOperator{\Map}{Map}
\DeclareMathOperator{\hMap}{hMap}

\DeclareMathOperator{\ob}{ob}

\DeclareMathOperator{\sk}{sk}

\DeclareMathOperator{\id}{id}

\newcommand{\ra}{\rightarrow}                   % right arrow
\newcommand{\lra}{\longrightarrow}              % long right arrow
\newcommand{\la}{\leftarrow}                    % left arrow
\newcommand{\lla}{\longleftarrow}               % long left arrow
\newcommand{\llra}[1]{\stackrel{#1}{\lra}}      % labeled long right
\newcommand{\llla}[1]{\stackrel{#1}{\lla}}      % labeled long right
\newcommand{\we}{\llra{\sim}}                   % weak equivalence
\newcommand{\bwe}{\llla{\sim}}
\newcommand{\cof}{\rightarrowtail}              % cofibration
\newcommand{\fib}{\twoheadrightarrow}           % fibration
\newcommand{\bfib}{\twoheadleftarrow}           % fibration
\newcommand{\trfib}{\stackrel{\sim}{\longfib}}
\newcommand{\trcof}{\stackrel{\sim}{\cof}}
\newcommand{\btrfib}{\stackrel{{\sim}}{\longbfib}}

\newcommand{\inc}{\hookrightarrow}              % inclusion
\newcommand{\blank}{-}                          % A hyphen, as in
\newcommand{\ovcat}{\downarrow}

\newcommand{\bd}[1]{\partial\Delta^{#1}}

\newcommand{\adjoint}{\rightleftarrows}

\newcommand{\bdd}[1]{\partial\Delta^{#1}}

\newcommand{\del}[1]{\Delta^{#1}}

\newcommand{\he}{\simeq}

\newcommand{\rea}[1]{|{#1}|}             %geometric realization of #1

\newcommand{\ceck}[1]{\Cech(#1)}         %Cech complex for #1
\newcommand{\oceck}[1]{\Cech^{o}(#1)}    %Ordered Cech complex for #1
\newcommand{\oreal}[1]{\rea{\oceck{U}}}  %Realization of ordered Cech cplex
\newcommand{\creal}[1]{\rea{\ceck{U}}}   %Realization of the Cech complex
\newcommand{\Cech}{\check{C}}

\numberwithin{equation}{subsection}

\newenvironment{myequation}
  {\addtocounter{subsection}{1}\begin{eqnarray}}
  {\end{eqnarray}$\!\!$}

\newcommand{\cyl}{cyl}
\newcommand{\Ccyl}{C_{\cyl}}
\newcommand{\jC}{\mathfrak{C}}

\newcommand{\jCn}{\mathfrak{C}^{nec}}
\newcommand{\jCh}{\mathfrak{C}^{hoc}}

\newcommand{\Nec}{{\cat{N}ec}}

\DeclareMathOperator{\Spi}{Spi}

\DeclareMathOperator{\Wcofib}{Wcofib}
\DeclareMathOperator{\Wfib}{Wfib}
\DeclareMathOperator{\W}{W}

\newcommand{\nosee}[1]{}

\def\tn{\textnormal}

\def\id{\tn{id}}
\def\to{\ra}
\def\To{\xrightarrow}
\def\too{\longrightarrow}

\def\From{\xleftarrow}

\def\taking{\colon}
\def\sCat{{s\Cat}}

\def\cross{\times}

\def\bD{{\bf \Delta}}
\newcommand{\Adjoint}[4]{\xymatrix@1{#2 \ar@<.5ex>[r]^-{#1} & #3 \ar@<.5ex>[l]^-{#4}}}
\newcommand{\bpsSet}{\sSet_{*,*}}

\newcommand{\bbox}{\square}
\newcommand{\jbox}{\boxtimes}
\DeclareMathOperator{\Mor}{Mor}
\newcommand{\join}{\star}

\begin{document}

\title{Mapping spaces in quasi-categories}

\author{Daniel Dugger}
\author{David I. Spivak}

\address{Department of Mathematics\\ University of Oregon\\ Eugene, OR
97403}

\address{Department of Mathematics\\ University of Oregon\\ Eugene, OR
97403}

\email{ddugger@uoregon.edu}

\email{dspivak@uoregon.edu}

\begin{abstract}
We apply the Dwyer-Kan theory of homotopy function complexes in model
categories to the study of mapping spaces in quasi-categories.  Using
this, together with our work on rigidification 
from \cite{DS1}, we  give a streamlined proof of the
Quillen equivalence between quasi-categories and simplicial
categories.  Some useful material about relative
mapping spaces in quasi-categories is developed along the way.
\end{abstract}

\maketitle

\tableofcontents

\section{Introduction}\label{se:intro}
A quasi-category is a simplicial set $X$ with a certain extension
property, weaker than the classical Kan condition.  This property
ensures that the simplicial set behaves like a higher-dimensional
category; that is, $X$ can be thought of as something like a category with
$n$-morphisms for any $n\geq 1$.  Quasi-categories were introduced
briefly by
Boardman and Vogt \cite{BV}, and have since been studied by \cite{CP},
\cite{J1}, \cite{J2}, and \cite{L}, among others.

If $K$ is a quasi-category and $x$ and $y$ are $0$-simplices, it is
possible to construct a mapping space $K(x,y)$ which is a simplicial
set.  The trouble is that there are many different ways to exhibit
such mapping spaces, the different models being weakly equivalent but
not isomorphic.  No particular model is ideal for every application,
and so one must become versatile at changing back-and-forth.  In
\cite{L} Lurie uses several models, most prominantly the ones denoted
there $\Hom_K^R(x,y)$, $\Hom_K^L(x,y)$, and $\jC(K)(x,y)$.  In order
to understand the connections between these, Lurie develops a theory
of ``straightening and unstraightening functors,'' which is reasonably
complicated.

This paper has two main goals.  First, we explain how the
mapping spaces of quasi-categories fit into the well-understood theory
of homotopy function complexes in model categories \cite{DK3}.   The
latter technology immediately gives various tools for understanding
why different models for these mapping spaces are weakly equivalent.
Lurie's $\Hom_K^R(x,y)$ and $\Hom_K^L(x,y)$---as well as several other
useful models---fit directly into this framework, but 
$\jC(K)(x,y)$ does not.  Our second goal is to use these tools, together
with our work on $\jC(\blank)$ from \cite{DS1}, to give a direct proof
that the $\jC(K)$ mapping spaces are weakly equivalent to the
Dwyer-Kan models.  This in turn allows us to prove that the homotopy
theory of quasi-categories is Quillen equivalent to that of simplicial
categories, giving a new approach to this result of Lurie's \cite{L}.

\medskip

We now describe the contents of the paper in more detail.  Recall that
a
quasi-category is a simplicial set $X$ that satisfies the extension
condition for the so-called inner horn inclusions---the inclusions
$\Lambda^n_k \inc \Delta^n$ for $0<k<n$.  
It turns out that
there is a model category structure on $\sSet$ where the cofibrations
are the monomorphisms, the fibrant objects are the quasi-categories,
and where the weak equivalences are something we will call {\it Joyal
equivalences}.  See Section~\ref{se:quasi} for more background.  We
will call this the {\it Joyal model structure\/}, and denote it
$\sSet_J$.  The classical weak equivalences of simplicial
sets---maps that induce homotopy equivalences of the geometric
realizations---will be termed {\it Kan equivalences\/} from now on.

In any model category $\cM$, given two objects $X$ and $Y$ there is a
homotopy function complex $\hMap_{\cM}(X,Y)$.  The theory of these
objects is largely due to Dwyer and Kan.  Such a function complex can be defined in
several ways, all of which are homotopy equivalent:
\begin{enumerate}[(1)]
\item It is the mapping space between $X$ and $Y$ in the simplicial
localization $L_W\cM$, where one inverts the subcategory $W$ of weak
equivalences \cite{DK1}.
\item It is the mapping space in the hammock localization $L_H\cM$
constructed by Dwyer and Kan in \cite{DK2}.
\item It can be obtained as the simplicial set $[n]\mapsto
\cM(Q^nX,\widehat{Y})$ where $Y\ra \widehat{Y}$ is a fibrant
replacement in $\cM$ and
$Q^\bullet X\ra X$ is a cosimplicial resolution of $X$ 
\cite{DK3}. 
\item It can dually be obtained as the simplicial set $[n]\mapsto
\cM(\tilde{X},R_n Y)$ where $\tilde{X}\ra X$ is a cofibrant
replacement and $Y\ra R_\bullet Y$ is a simplicial resolution of $Y$.
\item It can be obtained  as the nerve of various categories of
zig-zags, for instance (if $Y$ is fibrant) the category whose objects are zig-zags
\[ X \bwe A \lra Y \]
and where the maps are commutative diagrams equal to the identity on
$X$ and $Y$.
\end{enumerate}

In some sense the models in (3) and (4) are the most computable, but
one finds that all the models are useful in various situations.  
One learns, as a part of model category theory, how to compare these
different models and to see that they are Kan equivalent.  See
\cite{DK3} and \cite{D}, as well as Section~\ref{se:mapping-spaces} of the present
paper.

The above technology can be applied to the Joyal model structure in the
following way.  The overcategory $(\bd{1}\ovcat \sSet_J)$ inherits a
model category structure from $\sSet_J$ (\cite[7.6.5]{H}).  Given a
simplicial set $K$ with chosen vertices $a$ and $b$, consider $K$ as a
simplicial set under $\bd{1}$ via the evident map $\bd{1}\ra K$
sending $0\mapsto a$, $1\mapsto b$.  In particular, we can apply this to
$\del{1}$ and the vertices $0$ and $1$.  This allows us to consider
the homotopy function complex
\begin{myequation}
\label{hmap} 
\hMap_{(\bd{1}\ovcat \sSet_J)}(\del{1},K),
\end{myequation}
which somehow feels like the pedagogically `correct' interpretation of the mapping
space in $K$ from $a$ to $b$.  

The following result is more like an observation than a proposition 
(but it is restated and proved as Corollary~\ref{cor:four models}):

\begin{prop}\label{prop:observation}
Let $K$ be a quasi-category.  The mapping spaces $\Hom^R_K(a,b)$ and
$\Hom^L_K(a,b)$ of \cite{L} are models for the homotopy function
complex (\ref{hmap}), obtained via two different cosimplicial resolutions of
$\del{1}$.  The pullback
\[\xymatrix{
\Hom_K(a,b)\ar@{-->}[r]\ar@{-->}[d] & K^{\del{1}}\ar[d]\\
\Delta^0\ar[r]^-{(a,b)} & K\times K
}
\]
is also a  model for the same homotopy function complex, this time
obtained via a third cosimplicial resolution --- the one sending
$[n]$ to the pushout
$$\Delta^1\times\Delta^n\la\bd{1}\times\Delta^n\ra\bd{1}.$$
\end{prop}

Note that, given the above proposition, the Dwyer-Kan technology shows
immediately that the three constructions $\Hom^R_K(a,b)$,
$\Hom_K^L(a,b)$, and $\Hom_K(a,b)$ are all Kan equivalent, and in fact
gives a `homotopically canonical' weak equivalence between any two.

\subsection{Connections with the $\jC$ mapping spaces}\label{se:connections}
One problem with using the Dwyer-Kan models in our setting is that
given $0$-simplices $a$, $b$, and $c$ in a quasi-category $K$, there
do not exist naturally arising composition maps
\[ \hMap(\Delta^1,K_{b,c})\times \hMap(\Delta^1,K_{a,b}) \too
\hMap(\Delta^1,K_{a,c}).
\]
In \cite{L} Lurie gives a functor $\jC\colon \sSet \ra s\Cat$, and the
simplicial sets $\jC(K)(a,b)$ can be thought of as models for function
complexes of $K$ which {\it do\/} admit such composition maps.  In
\cite{DS1} we gave another functor $\jCn$ which accomplished the same
task, and we proved that the mapping spaces $\jC(K)(a,b)$ and
$\jCn(K)(a,b)$ were naturally Kan equivalent.  In fact we gave an
entire collection of different models $\jC^{\kG}$, all of which were
Kan equivalent.  What is not immediately clear is how to connect
the Dwyer-Kan function complexes to the mapping spaces arising in
 $\jC$, $\jCn$, or $\jC^{\kG}$.  
This is
explained in Section~\ref{se:connect}, where they are shown to be
connected by a canonical zig-zag of
Kan equivalences.  

At one level the connection can be seen as follows. Recall from
\cite{DS1} that a {\it necklace} is a simplicial set of the form
$\Delta^{n_1}\Wedge \cdots \Wedge \Delta^{n_k}$, obtained from
simplices by gluing the final vertex of one to the initial vertex of
its successor.  There is a natural linear order on the vertices of a necklace, and there is a unique 1-simplex connecting each non-terminal vertex to its successor.  The spine $\Spi[T]$ of a necklace $T$ is the union of these 1-simplices, and the associated simplex $\Delta[T]$ of $T$ is the simplex
with the same ordered vertex set as $T$.

The mapping space $\jCn(K)(a,b)$ is just the
nerve of the evident category whose objects are pairs $[T,T\ra K_{a,b}]$
where $T$ is a necklace and $T\ra K$ is a map sending the initial and
final vertices of $T$ to $a$ and $b$.  The mapping spaces
$\jC^{\kG}(a,b)$ are defined similarly, but where one replaces the
category of necklaces with a category $\kG$ of some other suitable ``gadgets."

For any necklace $T$, there is a canonical inclusion
$T\inc \Delta[T]$; it may be checked that this inclusion is a weak equivalence in
$\sSet_J$ (see Lemma \ref{le:necklace=simplex}).  Any map $T\ra K$ therefore gives rise to a zig-zag
\[ \del{1} \lra \Delta[T] \bwe T \ra K \]
where the map $\del{1}\ra \Delta[T]$ is the unique $1$-simplex
connecting the initial and final vertices of the simplex $\Delta[T]$.
Zig-zags of the above type are known to give a model for homotopy
function complexes (see Section \ref{se:mapping-spaces}).

The considerations from the above paragraph lead to a comparison map
between the $\jC(K)(a,b)$ mapping spaces and the Dwyer-Kan mapping
spaces.  This map turns out to be a Kan equivalence, although it
should be noted that our proof of this is not direct: it is a crafty
argument using one of the $\jC^{\kG}$ constructions where one replaces
necklaces by a more general class of gadgets.  See
Section~\ref{se:connect} for the details.  In any case, we obtain the
following (where one may choose any standard model for the homotopy
function complex):

\begin{thm}
\label{th:main1}
Given a quasi-category $K$ with objects $a$ and $b$, there is a
natural zig-zag of Kan equivalences between the simplicial sets
$\jC(K)(a,b)$ and $\hMap(\Delta^1,K_{a,b})$.
\end{thm}

The previous theorem is a key step in the proof of the following
important result:

\begin{thm}
\label{th:main2}
The adjoint pair $\jC\colon \sSet_J\adjoint s\Cat\colon N$ is a
Quillen equivalence between the Joyal model structure on $\sSet$ and
the model structure on $s\Cat$ due to Bergner \cite{Bergner}.
\end{thm}

We prove the above theorem as Corollary~\ref{cor:main theorem}.

\subsection{Relation with the work of Lurie}
Both Theorem~\ref{th:main1} and \ref{th:main2} are originally due to
Lurie, and are proven in \cite{L}.  In Lurie's book he starts by
developing the properties of mapping spaces and rigidification and
then proves the existence of the Joyal model structure as a
consequence of this work.  His methods involve a detailed and lengthy
study of what he calls ``straightening and unstraightening'' functors,
and it was a vague dissatisfaction with this material---together with
the hope of avoiding it---that first led us to the work in the present
paper.

In the present paper we {\it start\/} with the Joyal model structure.
In both \cite{L} and \cite{J2} it takes over a hundred pages to prove
its existence, but by boiling things down to the bare essentials
(and following the method of Joyal) we are able to give a streamlined
approach and create the structure fairly quickly.  At that point we
immediately have the Dwyer-Kan techniques at our disposal, and it is
through the use of these that we develop the properties of mapping
spaces and rigidification.  Thus, in some ways our approach is
opposite that of \cite{L}.

Due to the inherent differences in the two approaches, it is slightly
awkward for us to quote results from \cite{L} without creating
confusions and possible circularities.  Because of this, there are a
few minor results  whose proofs we end up repeating or
redoing in a slightly different way.  The result is that the present
paper can be read independently of \cite{L}---although this should not
be taken as a denial of the  intellectual debt we owe to that
work.

\subsection{Organization of the paper}

The material on quasi-categories needed in the paper is all reviewed
in Section~\ref{se:quasi}.  This material is mostly due to Joyal
\cite{J2}, and can also be found in Lurie \cite{L}.  Despite this, we
have included three appendices in the paper where we prove all the
assertions from Section~\ref{se:quasi}.   The reason for this is that
although the results we need are reasonably straightforward, in both
\cite{L} and \cite{J2} they are intertwined with so many other things
that it is difficult for a reader to extract their proofs.  Especially
in the case of the Joyal model structure, a basic tool in our entire
approach, we felt that it was important to have a relatively short and
self-contained argument.  The fourteen pages in our appendices are a
bit longer than ideal, but quite a bit shorter than the currently
available alternatives. 

Let us outline the rest of the paper.  In Section \ref{se:mapping-spaces}, we review homotopy function complexes in general model categories $\cM$.  We apply these ideas in Section \ref{se:Dwyer-Kan} to the case $\cM=(\bd{1}\ovcat\sSet_J)$.  That is, the function complex of maps $\Delta^1\to X$ in $\cM$ gives the ``internal mapping space" between vertices $a$ and $b$ in a quasi-category $X$.  In Section \ref{se:connect} we compare these internal mapping spaces to the rigidification mapping spaces $\jC(X)(a,b)$, and in so doing rely heavily on \cite{DS1}.  In Section \ref{se:relative} we define relative mapping spaces (the map $\bd{1}\to\Delta^1$ is the thing being generalized here).  In Section \ref{se:DK-equiv} we introduce the notion of DK-equivalence in $\sSet$, which is somehow intermediate between $\jC$-equivalence and Joyal equivalence.  In Section \ref{se:Qequiv} we prove that these three notions of equivalence agree.  Section \ref{se:Qequiv} also provides a new proof of Lurie's result that there is Quillen equivalence between $\sSet_J$ and $\sCat$ (see Corollary \ref{cor:main theorem}).  Finally, in Section \ref{se:leftover} we prove some technical results that are stated without proof in Sections \ref{se:intro} and \ref{se:Dwyer-Kan}.

\subsection{Notation and Terminology}\label{se:notation}
We will use $\sSet_K$ to refer to the usual model structure
on simplicial sets, which we'll term the {\it Kan model structure\/}.
The fibrations are the Kan fibrations, the weak equivalences (called
Kan equivalences from now on) are the maps which induce homotopy
equivalences on geometric realizations, and the cofibrations are the
monomorphisms.  

Quasi-categories are simplicial sets, and the mapping spaces between their vertices are also simplicial sets. Quasi-categories will always be viewed within the Joyal model structure, and their mapping spaces will always be viewed within the Kan model structure.  

We will often be working with the category $\bpsSet=(\bdd{1}\ovcat\sSet)$.
When we consider it as a model category, the model structure we use will always
be the one imported from the Joyal model structure on $\sSet$; we will denote it $(\bpsSet)_J=(\bdd{1}\ovcat\sSet_J)$, or just $\bpsSet$.   When we write $\Delta^1$ as an object of $\bpsSet$, we always mean to use the canonical inclusion $\bd{1}\to\Delta^1$.

An object of $\bpsSet$ is a simplicial set $X$ with two distinguished
points $a$ and $b$.  We sometimes (but not always) write $X_{a,b}$ for $X$, to remind
ourselves that things are taking place in $\bpsSet$ instead of $\sSet$. 

When $\cC$ is a category we write $\cC(a,b)$ instead of
$\Hom_{\cC}(a,b)$.  We reserve the $\Hom_K(a,b)$ notation for when $K$ is a quasi-category, to denote various models of the homotopy function complex.  See Proposition \ref{prop:observation} or Section \ref{se:cosimplicial versions}.

Finally, in certain places we have to deal with
simplices and horns.  The horn $\Lambda^n_k$ is the 
 union of all faces containing
the vertex $k$.  If $\{m_0,\ldots,m_k\}$ is an ordered
set, we use $\Delta^{\{m_0,\ldots,m_k\}}$ as notation for the
$k$-simplex whose vertices are $m_0,\ldots,m_k$.  Likewise, the
``$m_i$-face'' of $\Delta^{\{m_0,\ldots,m_k\}}$ is the codimension one
face not containing the vertex $m_i$, and 
$\Lambda^{\{m_0,\ldots,m_k\}}_{m_i}$ denotes the horn consisting of
all faces that contain  $m_i$.  
If $j_0\leq\cdots\leq j_n$, we may use ``bracket notation" and write $[m_{j_0},\ldots,m_{j_n}]$ for the subsimplex
of $\Delta^{\{m_0,\ldots,m_k\}}$ having vertices
$m_{j_0},\ldots,m_{j_n}$. 

%%%%%%%%%%%%%%%%%%%%%%%%%%%%%%%%%%%%%%%%%%%%%%%%%%%%%%%%%%%%%%%%%%%%%

\section{A quick introduction to quasi-categories}
\label{se:quasi}

In this section we give an account of the basic properties of
quasi-categories, including the Joyal model structure.  Almost all of
this material is due to Joyal \cite{J2}, and is also buried deep in
\cite{L}.  A few simple proofs are included here, but the longer proofs
are not.  The reader is referred to the appendices for a guide through
these longer
proofs.

\medskip

\subsection{Isomorphisms in quasi-categories}
For a quasi-category $X$, we refer to the $0$-simplices $a\in X_0$ as
\dfn{objects}; to the $1$-simplices $f\in X_1$ as \dfn{morphisms},
which we
denote by $f\taking a\to b$ if $d_1(f)=a, d_0(f)=b$; and to the
degenerate 1-simplices $s_0(a)$ as \dfn{identity morphisms}, which we
denote by $\id_a$.

For any set $S$, let $ES$ be the $0$-coskeleton of the set $S$.  Write
$E^n$ for $E(\{0,1,\ldots,n\})$.  Note that $E^1$ has exactly two
non-degenerate simplices in each dimension, and that its geometric
realization is the usual model for $S^\infty$.

Let $X$ be a quasi-category.
An \dfn{\mdfn{$\infty$}-isomorphism} in $X$ is a map $E^1\ra X$.  We
sometimes abuse terminology and say that a $1$-simplex $\Delta^1\ra X$
is an $\infty$-isomorphism if it extends to a map $E^1\ra X$.  

There is a closely related concept called \dfn{quasi-isomorphism}; a quasi-isomorphism
is a map $\sk_2(E^1)\ra X$.  The simplicial set $\sk_2(E^1)$ consists
of two nondegenerate $1$-simplices and two nondegenerate
$2$-simplices.  In other words, a quasi-isomorphism in $X$
consists of two objects $a$ and $b$, maps $f\colon a\ra b$ and
$g\colon b\ra a$, and two $2$-simplices of the form
\[ \xymatrix{
a\ar[r]^f \ar[dr]_{\id_a} & b \ar[d]^g && b \ar[r]^g \ar[dr]_{\id_b} &
a \ar[d]^f \\
& a && & b.
}
\]
Again, we sometimes abuse terminology and say that a $1$-simplex in
$X$ is a quasi-isomorphism if it extends to a map $\sk_2(E^1)\ra X$.  

The following results record the key properties we will need concerning
$\infty$-isomorphisms and quasi-isomorphisms:

\begin{prop}
\label{pr:quasi0}
Let $X$ be a quasi-category, and let $f\colon \Delta^1\ra X$.  The
following conditions on $f$ are equivalent:
\begin{enumerate}[(i)]
\item $f$ is an $\infty$-isomorphism;
\item $f$ is a quasi-isomorphism;
\item $f$ has a left quasi-inverse and a (possibly different) right
quasi-inverse.  That is, there exist $2$-simplices in $X$ of the form
\[ \xymatrix{
a\ar[r]^f \ar[dr]_{\id_a} & b \ar[d]^{g_1} \\
& a} \qquad\text{and}\qquad
\xymatrix{
 b \ar[r]^{g_2} \ar[dr]_{\id_b} & a \ar[d]^f \\
 & b.
}
\] 
\end{enumerate}
\end{prop}

\begin{proof}
See ``Proof of Proposition \ref{pr:quasi0}" in Appendix~\ref{se:special}.
\end{proof}

\begin{prop} 
\label{pr:quasi1}
Let $X$ be a quasi-category.  Then the map $X\ra *$ has the
right-lifting-property with respect to the following:
\begin{enumerate}[(a)]
\item The maps
$(A\times E^1)\amalg_{(A\times
\{0\})} (B\times \{0\}) \inc B\times E^1$, for every monomorphism
$A\inc B$.
\item The 
inclusion $E(\{0,1\})\cup E(\{1,2\})\inc E(\{0,1,2\})$.  
\end{enumerate}
\end{prop}

Note that part (b) implies that one may compose (in a non-unique way)
$\infty$-isomorphisms to yield another $\infty$-isomorphism.

\begin{proof}
Part (a) is a combinatorial argument which we postpone (see Appendix~\ref{se:special}).  However, (b) is an easy consequence of (a), using the maps
$j\colon E^2\ra E^1\times E^1$ and $p\colon E^1\times E^1\ra E^2$
defined by:
\[ j(0)=(0,0), \quad j(1)=(1,0), \quad j(2)=(1,1) \]
and
\[
p(0,0)=p(0,1)=0, \quad p(1,0)=1, \quad p(1,1)=2.
\]
(Both $E^2$ and $E^1\times E^1$ are $0$-coskeleta, so a map into them
is completely determined by what it does to $0$-simplices).  
The maps $j$ and $p$ exhibit the inclusion from (b)
as a retract of
$ (\{0,1\}\times E^1)\amalg_{\{0,1\}\times \{0\}} (E^1\times \{0\})
\inc E^1\cross E^1$.
\end{proof}

The proofs of Proposition~\ref{pr:quasi0} and \ref{pr:quasi1}(a) are
not at all straightforward.  They hinge on a lemma discovered by Joyal
which we also record here (although the proof is again deferred to
Appendix~\ref{se:special}).

\begin{prop}[Joyal, Special outer horn lifting]
\label{pr:quasi2}
Let $X$ be a quasi-category.   Given either of the following
lifting diagrams
\[ \xymatrix{
\Lambda^n_0\ar[r]^f\ar[d] & X &&& \Lambda^n_n\ar[r]^g\ar[d] & X\\
\Delta^n\ar@{.>}[ur] &&&& \Delta^n \ar@{.>}[ur]
}
\]
in which, respectively, $f([01])$ or $g([n-1,n])$ is a quasi-isomorphism,
there is a lifting as shown.
\end{prop}

It is easy to see that in a Kan complex every morphism is a
quasi-isomorphism (use Proposition~\ref{pr:quasi0}).  Conversely,
if in a certain quasi-category $X$ every morphism is a
quasi-isomorphism then Proposition~\ref{pr:quasi2} implies that $X$ is
a Kan complex.  Thus one has the slogan ``Kan complexes are
$\infty$-groupoids.''

For a quasi-category $X$, let $J(X)\subseteq X$ be the subcomplex
consisting of those simplices having the property that every 1-dimensional face
is a quasi-isomorphism.  If $\Kan\subseteq \sSet$ denotes the full
subcategory of Kan complexes, then $J\colon \sSet\ra \Kan$ and this is
the right adjoint of the inclusion functor $\Kan\inc \sSet$.  In
particular, note that $J$ preserves limits.  The quasi-category $J(X)$
is the ``$\infty$-groupoid of $\infty$-isomorphisms in $X$.''

\subsection{Box product lemmas}

A map of simplicial sets will be called \dfn{inner
anodyne} if it can be constructed out of the maps $\Lambda^n_k\inc
\Delta^n$, $0<k<n$, by cobase changes and compositions.  Clearly any
quasi-category has the right-lifting-property with respect to all
inner anodyne maps.  

Given morphisms $f\colon A\ra B$ and $g\colon C\ra D$, let $f\bbox g$ denote
the induced map 
\[ (A\times D)\amalg_{(A\times C)} (B\times C) \lra
B\times D. 
\]

\begin{prop}[Joyal]
\label{pr:box-inner}
If $i\colon \Lambda^n_k\inc \Delta^n$ is an inner horn inclusion and
the map $j\colon A\to B$ is any monomorphism, then $i\bbox
j$ is inner anodyne.
\end{prop}

\begin{proof}
See Appendix~\ref{se:box}.
\end{proof}

A map $X\ra Y$ is an \dfn{inner fibration} if it has the
right-lifting-property with respect to the inner horn inclusions.  The
following result is an immediate consequence of
Proposition~\ref{pr:box-inner}, using adjointness:

\begin{prop}
\label{pr:matching-fibration}
If $A\cof B$ is a monomorphism and $X\ra Y$ is an inner fibration,
then $X^B\ra X^A\times_{Y^A} Y^B$ is also an inner fibration.  In
particular, if $X$ is a quasi-category then so is $X^A$ for any
simplicial set $A$.  
\end{prop}

\begin{proof}
Immediate.
\end{proof}

\subsection{Joyal equivalences}
Let $A$ and $Z$ be simplicial sets, and assume $Z$ is a
quasi-category. Define a relation on $\sSet(A,Z)$ by saying that
$f\sim g$ if there exists a map $H\colon A\times E^1 \ra Z$ such that
$Hi_0=f$ and $Hi_1=g$, where $i_0,i_1 \colon A\inc A\times E^1$ are
the evident inclusions.  It is easy to see that this relation is both
reflexive and symmetric, and Proposition~\ref{pr:quasi1}(b) implies
that it is also transitive (using that $Z^A$ is a quasi-category).  We
write $[A,Z]_{E^1}$ to denote the set $\sSet(A,Z)/\sim$.

\begin{defn}
A map of simplicial sets $A\ra B$ is a \dfn{Joyal equivalence} if
$[B,Z]_{E^1}
\ra [A,Z]_{E^1}$ is a bijection for every quasi-category $Z$.  A
\dfn{Joyal acyclic cofibration} is a map which is both a monomorphism
and a Joyal equivalence.
\end{defn}

%\begin{prop}
%\label{pr:E1}
%The inclusions $\{0\}\inc E^1$ and $\{1\}\inc E^1$ are Joyal equivalences.
%\end{prop}
%
%\begin{proof}
%As the maps are isomorphic, it is enough to prove the result for the
%first map.  We muse show that for any quasi-category $X$ the map
%$[E^1,X]\ra [\{0\},X]$ is injective (as surjectivity is obvious).
%
%Suppose given $f,g\colon E^1\ra X$ and a map $h\colon E^1\ra X$ such
%that $h(0)=f(0)$ and $h(1)=g(0)$.  Then putting $f$, $g$, and $h$
%together we have a map
%\[ 
% [\{0,1\}\times E^1]
% \amalg_{[\{0,1\}\times \{0\}]}
%[E^1 \times \{0\}]
%\ra X.
%\]
%By Proposition~\ref{pr:quasi1}(a) this extends to a map $E^1\times E^1\ra
%X$, and this is the desired $E^1$-homotopy between $f$ and $g$.
%\end{proof}

\begin{remark}
\label{re:E1-homotopy}
We define a map of simplicial sets $f\colon X\ra Y$ to be an
\mdfn{$E^1$}\dfn{-homotopy equivalence} if there is a map $g\colon
Y\ra X$ such that $fg$ and $gf$ are $E^1$-homotopic to their
respective identities.  It is trivial to check that every
$E^1$-homotopy equivalence is a Joyal equivalence, and that if $X$
and $Y$ are quasi-categories then a map $X\ra Y$ is a Joyal
equivalence if and only if it is an $E^1$-homotopy equivalence.
Note that $\{0\}\inc E^1$ is readily seen to be an $E^1$-homotopy equivalence.
\end{remark}

\begin{prop}\label{pr:first-properties} 
The following statements are true:
\begin{enumerate}[(a)]
\item
If $C\ra D$ is a Joyal acyclic cofibration then so is $A\times C\ra
A\times D$, for any simplicial set $A$.
\item If $X$ is a quasi-category, then $X\ra *$ has the
right-lifting-property with respect to every Joyal acyclic
cofibration.
\item Every inner horn inclusion $\Lambda^n_k \inc \Delta^n$, $0<k<n$,
is a Joyal equivalence.  
\end{enumerate}
\end{prop}

\begin{proof}
For (a) we must show that $[A\times D,X]_{E^1} \ra [A\times C,X]_{E^1}$ is a
bijection for every quasi-category $X$.  But this map is readily
identified with $[D,X^A]_{E^1}\ra [C,X^A]_{E^1}$, which is a bijection
because $X^A$ is also a quasi-category
(Proposition~\ref{pr:matching-fibration}).

Part (b) is an easy exercise using the definitions and
Proposition~\ref{pr:quasi1}(a).

For part (c), if $X$ is a quasi-category then $[\Delta^n,X]\ra
[\Lambda^n_k,X]$ is surjective by the definition of quasi-category.
Injectivity follows from Proposition~\ref{pr:box-inner}, which says
that $X\ra *$ has the RLP with respect to $(\Lambda^n_k\inc \Delta^n)\bbox
(\{0,1\}\inc E^1)$.
\end{proof}

\subsection{The model category structure}

Like all the results in this section, the following is due to Joyal.  It is proven in Appendix \ref{se:model}.

\begin{thm}
\label{th:model}
There exists a unique model structure on $\sSet$ in which the
cofibrations are the monomorphisms and the fibrant objects are the
quasi-categories.  The weak equivalences in this structure are the
Joyal equivalences, and the model structure is cofibrantly-generated.
\end{thm}

The model category structure provided by Theorem~\ref{th:model} will
be denoted $\sSet_J$.  Note that it is left proper because every object is cofibrant.
We will use the term {\it Joyal fibration\/} for the  fibrations in $\sSet_J$. 

\begin{remark}
It should be noted that while $\sSet_J$ is cofibrantly-generated, no
one has so far managed to write down a manageable set of generating
acyclic cofibrations. 
% For example, the inner horn inclusions together
%with $\{0\}\to E^1$ do not suffice.  
%\shortnote{Is this obvious?}
\end{remark}

The following result will be used frequently:

\begin{prop}
\label{pr:Joyal-match}
Let $f\colon A\cof B$  and $g\colon
C\cof D$ be cofibrations in $\sSet_J$, and let $h\colon X\ra Y$ be a Joyal
fibration.  Then the following statements are true:
\begin{enumerate}[(a)]
\item $f\bbox g$ is a cofibration, which is Joyal acyclic if $f$ or $g$
is so.  
\item  $X^B\ra (Y^B\times_{Y^A} X^A)$ is a Joyal fibration, which is acyclic
if $f$ or $h$ is so.
\end{enumerate}
\end{prop}

\begin{proof}
Part (b) follows immediately from (a) by adjointness.  The statement
in (a) that $f\bbox g$ is a cofibration is evident, so we must only
prove that it is a Joyal equivalence if $f$ is (the case when $g$ is a
Joyal equivalence following by symmetry).
Consider the diagram
\[ \xymatrix{
A\times C \ar@{ >->}[r]\ar@{ >->}[d]_\sim & A\times D\ar@{ >->}[d]\ar[ddr]^\sim \\
B\times C \ar@{ >->}[r]\ar[rrd] & P \ar@{.>}[dr]\\
&&B\times D
}
\]
where $P$ is the pushout.
The labelled maps are Joyal equivalences by Proposition~\ref{pr:first-properties}(a).  It follows
that the pushout map $A\times D \ra P$ is a Joyal acyclic cofibration,
and hence $P\ra B\times D$ is a Joyal equivalence by two-out-of-three.
\end{proof}

\subsection{Categorification and the coherent nerve}
We recall from \cite{L} that there are adjoint functors
\[ \jC \colon \sSet\adjoint \sCat \colon N.
\]
Here $\jC(\Delta^n)$ is a specific simplicial category one writes
down, and $N(\cD)$ is the simplicial set $[n]\mapsto
\sCat(\jC(\Delta^n),\cD)$.  See \cite{DS1} for more information, as
well for a  very detailed description of the functor $\jC$.  We call
$\jC$ the {\it categorification\/} functor, and $N$ the {\it
coherent nerve\/}.

It is useful to also consider the composite of adjoint pairs
\[ \xymatrix{
\sSet \ar@<0.5ex>[r]^{\jC} & \sCat
\ar@<0.5ex>[r]^{\pi_0}\ar@<0.5ex>[l]^N 
& \Cat. \ar@<0.5ex>[l]^{c}
}
\]
where the left adjoints all point left to right.
Here $c$ is the functor which regards a category as a discrete
simplicial category, and $\pi_0$ is the functor which applies the
usual $\pi_0$ to all the mapping spaces.  Note that the composite $Nc$
is the classical nerve functor; as $\pi_0\jC$ is its
left adjoint, this identifies $\pi_0\jC$ as the functor which takes a
simplicial set and forms the free category from the $0$-simplices and
$1$-simplices, imposing composition relations coming from the
$2$-simplices.  This functor is denoted $\tau_1$ in \cite{J2}, but we
will always call it $\pi_0\jC$ in the present paper.  

\begin{remark}
For a general simplicial set $X$, a map in $\pi_0\jC(X)$ from $a$ to $b$
is a formal composite of maps represented by $1$-simplices in $X$.
But in a quasi-category we may ``compose'' adjacent $1$-simplices by
filling an inner horn of dimension $2$, and this shows that every map in
$\pi_0\jC(X)$ is represented by a single $1$-simplex of $X$.  
\end{remark}

The following result is due to Joyal:

\begin{prop}
\label{pr:tau_1}
Let $X$ be a quasi-category, and let $f$, $g$, and $h$ be
$1$-simplices in $X$.  Then one has $h=g\circ f$ in $\pi_0\jC(X)$ if
and only if there exists a map $\Delta^2\ra X$ whose $0$-, $1$-, and
$2$-faces are $g$, $h$, and $f$, respectively.
\end{prop}

The proof will be given in Appendix~\ref{se:special}.  The following is an
immediate corollary:

\begin{cor}
\label{co:3iso}
Let $f\colon \Delta^1\ra X$ be a $1$-simplex in a quasi-category.
Then $f$ is a quasi-isomorphism if and only if
the image of $f$ in the category $\pi_0\jC(X)$ is an isomorphism. 
\end{cor}

\begin{proof}
Immediate from Proposition~\ref{pr:tau_1}.
\end{proof}

Note that Corollary~\ref{co:3iso}
adds another equivalent condition to the list given in
Proposition~\ref{pr:quasi0}.  Also observe that this corollary
implies a two-out-of-three property for quasi-isomorphisms: for a
$2$-simplex in $X$, if two of its three faces are quasi-isomorphisms then so is the third.

We now also have the following:

\begin{prop}
\label{pr:isoclasses}
For $X$ a quasi-category, there is a natural bijection between
$[*,X]_{E^1}$ and the set of isomorphism classes of objects in $\pi_0\jC(X)$.
\end{prop}

\begin{proof}
First note that $\pi_0\jC(E^1)$ is the category consisting of two
objects and a unique isomorphism between them.  So it is clear that we
get a natural map $[*,X]_{E^1} \ra \isoc(\pi_0\jC(X))$, and that this
is surjective.  

Suppose two $0$-simplices $a,b\in X_0$ are isomorphic
in $\pi_0\jC(X)$.  
There is a $1$-simplex $e$ connecting $a$
and $b$ representing this isomorphism, and by Corollary~\ref{co:3iso}
$e$ is also a quasi-isomorphism.  By Proposition~\ref{pr:quasi1},
$e$ extends to a map $E^1\ra X$.  This shows that $a$ and $b$
are equal in $[*,X]_{E^1}$, as desired.  
\end{proof}

%%%%%%%%%%%%%%%%%%%%%%%%%%%%%%%%%%%%%%%%%%%%%%%%%%%%%%%%%%%%%%%%%%%%%%%%%%%%%%

\section{Background on mapping spaces in model categories}
\label{se:mapping-spaces}

Given two objects $X$ and $Y$ in a model category $\cM$,
there is an associated simplicial set  $\hMap_\cM(X,Y)$
called a ``homotopy function complex'' from $X$ to $Y$.  The basic
theory of these function complexes is due to Dwyer-Kan \cite{DK1,DK2,DK3}.  
As recounted in the introduction, there are several different ways to
write down models for these function complexes, all of which turn out
to be Kan equivalent.  In this section we give a brief review of
some of this machinery.

\subsection{Mapping spaces via cosimplicial resolutions}
Let $\cM$ be a model category, and let $c\cM$ be the Reedy model
category of cosimplicial objects in $\cM$ \cite[Chapter 15]{H}.  For
any $X\in \cM$ we will write $cX$ for the constant cosimplicial object
consisting of $X$ in every dimension, where every coface and
codegeneracy is the identity.

If $X\in \cM$, a \dfn{cosimplicial resolution} of $X$ is a Reedy cofibrant
replacement $Q^\bullet \we cX$.    Given such a cosimplicial
resolution and an object $Z\in \cM$, we may form the simplicial set
$\cM(Q^\bullet,Z)$ given by
\[ [n]\mapsto \cM(Q^n,Z).
\]
It is known \cite[16.5.5]{H} that if $Z\ra Z'$ is a weak equivalence between
fibrant objects then the induced map $\cM(Q^\bullet,Z)\ra
\cM(Q^\bullet,Z')$ is a Kan equivalence of simplicial sets.

%%%%%%%%%%%%%%%%%%%%%%%%%%%%%%%%%%%%%%%%%%%%%%%%%%%%%%%%%%%%%%%%%%%%%%%%%%%%%%

\subsection{Mapping spaces via nerves of categories}

For any object $X\in \cM$, let
$\cQ(X)$ be the category whose objects are pairs $[Q,Q\ra X]$ where
$Q$ is cofibrant and $Q\ra X$ is a weak equivalence.  For any object
$Y\in \cM$, there is a functor
\[ \cM(\blank,Y)\colon \cQ(X)^{op}\lra \Set \]
sending $[Q,Q\ra X]$ to $\cM(Q,Y)$.  
We can regard this functor as taking values in $\sSet$ by composing
with the embedding $\Set\inc \sSet$.  

Consider the simplicial set $\hocolim_{\cQ(X)^{op}} \cM(\blank,Y)$.  We fix
our model for the hocolim functor to be the result of first taking the
simplicial replacement of a diagram and then applying geometric
realization.  Notice in our case that in dimension $n$ the simplicial
replacement consists of diagrams of weak equivalences $Q_0\From{\sim} Q_1
\From{\sim}\cdots \From{\sim} Q_n$ over $X$ (with each $Q_i\ra X$ in $\cQ(X)$), 
together with a map $Q_0\ra Y$.  This
shows that the simplicial replacement is nothing but the nerve of the
category for which an object is a zig-zag $[X\bwe Q \ra Y]$, where $Q$ is
cofibrant and $Q\ra X$ is a weak equivalence; and a map from $[X\bwe Q \ra
Y]$ to $[X\bwe Q' \ra Y]$ is a map $Q'\ra Q$ making the evident
diagram commute.

Categories of zig-zags like the one considered above were first
studied in \cite{DK3}.  There are many variations, and it is basically
the case that all sensible variations have Kan equivalent nerves;
moreover, these nerves are Kan equivalent to the homotopy function
complex $\hMap_\cM(X,Y)$ (defined to be the space of maps in the
simplicial localization of $\cM$ with respect to the weak
equivalences).  We will next recall some of this machinery.  In
addition to \cite{DK3}, see \cite{D}.

Following \cite{DK3}, write $(\Wcofib)^{-1}\cM(\Wfib)^{-1}(X,Y)$ to
denote the category whose objects are zig-zags
\[ \xymatrixcolsep{1.3pc}\xymatrix{
X & U \ar[r]\ar@{->>}[l]_\sim & V & Y,\ar@{ >->}[l]_\sim 
}
\]
and where the maps are natural transformations of diagrams which are
the identity on $X$ and on $Y$.
Similarly, let $W^{-1}\cM W^{-1}(X,Y)$ be the category whose objects
are zig-zags
\[ X \bwe U \lra V \bwe Y, \]
let $\cM(\Wfib)^{-1}\cM(X,Y)$ be the category whose objects are
zig-zags
\[ X \lra U \btrfib V \lra Y,
\]
and so on.

Note that there are natural inclusions of two types: an example of the
first is $\cM(\Wfib)^{-1}(X,Y)\inc \cM \W^{-1}(X,Y)$ (induced by
$\Wfib\inc \W$), and an example
of the second is
\[ \cM \W^{-1}(X,Y) \inc \cM \W^{-1}\cM(X,Y) \]
which sends 
\[ [X \bwe A \lra Y] \mapsto [X\llra{\id} X \bwe A \lra Y].
\]

The following proposition is a very basic one in this theory, and will
be used often in the remainder of the paper; we have
included the proof for completeness, and because it is simple.  

\begin{prop}\label{prop:DK-1}
When $Y$ is fibrant, the maps in the following commutative square all
induce Kan equivalences on nerves:
\[ \xymatrix{
  \cM(\Wfib)^{-1}(X,Y) \ar@{ >->}[r]\ar@{ >->}[d] & \cM
(\Wfib)^{-1}\cM(X,Y).\ar@{ >->}[d]\\
\cM W^{-1}(X,Y) \ar@{ >->}[r] & \cM W^{-1}\cM(X,Y) 
}
\]
\end{prop}

\begin{proof}
Denote all the inclusions in the square by $j$.

We start with the left vertical map.  Given a zig-zag $[X\bwe A \lra
Y]$, functorially factor the map $A\ra X\times Y$ as $A\trcof P \fib
X\times Y$.  Since $Y$ is fibrant the projection $X\times Y\ra X$ is
a fibration, and so the composite $P\ra X$ is a fibration as well.
Define a functor $F\colon \cM \W^{-1}(X,Y) \ra \cM (\Wfib)^{-1}(X,Y)$ by
sending $[X\bwe A \lra Y]$ to $[X \btrfib P \lra Y]$.  There are natural
transformations $\id \ra j\circ F$ and $\id \ra F\circ j$, which shows
that on nerves $F$ and $j$ are homotopy inverses.

A very similar proof works for the right vertical map in the diagram.  Given a
zig-zag $[X\lra U \bwe V \lra Y]$, functorially factor $V\ra U\times
Y$ as $V\trcof P \fib U\times Y$.  Define $F\colon \cM W^{-1}\cM(X,Y)
\ra \cM (\Wfib)^{-1} \cM(X,Y)$ by sending the zig-zag 
\[ [X\lra U \bwe V \lra Y] \mapsto
[X\lra U \btrfib P \lra Y].
\]
  This gives a homotopy inverse for $j$.

For the top horizontal map we do not even need to use that $Y$ is
fibrant.  Define a homotopy inverse by sending
$[X\lra U \btrfib V \lra Y]$ to the associated zig-zag $[X \btrfib P \lra Y]$
where $P$ is the pullback of $X\lra U \btrfib V$.

Finally, the bottom horizontal map induces a Kan equivalence on nerves
because the other three maps do.
\end{proof}

Let $QX^\bullet\ra X$ be a cosimplicial resolution of $X$ in $c\cM$.
Following \cite{DK3}, we now relate the simplicial set
$\cM(QX^\bullet,Y)$ to the nerves of the categories of zig-zags
considered above.

For any simplicial set $K$, let $\bD K$ be the category of
simplices of $K$.  This is none other than the overcategory $(S\ovcat
K)$, where $S\colon \bD \ra \sSet$ is the functor $[n]\mapsto
\Delta^n$.  It is known that the nerve of $\bD K$ is naturally
Kan equivalent to $K$ (see \cite[text prior to Prop. 2.4]{D} for an
explanation).  

There is a functor $\bD \cM(QX^\bullet,Y) \ra \cM(\Wfib)^{-1}(X,Y)$
sending $([n],QX^n\ra Y)$ to $[X\btrfib QX^n \lra Y]$.  

\begin{prop}\label{prop:DK-2} Let $QX^\bullet\ra X$ be a Reedy cofibrant resolution of $X$.
Then
$\bD\cM(QX^\bullet,Y) \ra \cM(\Wfib)^{-1}(X,Y)$ induces a Kan
equivalence on nerves.
\end{prop}

\begin{proof}
The result is proven in \cite{DK3}, but see also
\cite[Thm. 2.4]{D}.
\end{proof}

\begin{remark}
\label{re:DK-summary}
To briefly summarize the main points of this section, Propositions \ref{prop:DK-1} and \ref{prop:DK-2} show that
when $X$ is any object and $Y$ is a fibrant object in a model category $\cM$, the following maps of categories all induce Kan
equivalences on the nerves:
\[ \xymatrixcolsep{1.4pc}\xymatrix{
\bD\cM(QX^\bullet,Y) \ar[r] & \cM(\Wfib)^{-1}(X,Y)\ar@{ >->}[r] & \cM\W^{-1}(X,Y) \ar@{ >->}[r]
& \cM\W^{-1}\cM(X,Y). \\
}
\]
In particular, the nerves all have the homotopy type of the homotopy
function complex $\hMap(X,Y)$.  
\end{remark}
%%%%%%%%%%%%%%%%%%%%%%%%%%%%%%%%%%%%%%%%%%%%%%%%%%%%%%%%%%%%%%%%%%%%%%%%%%%%%%%%%%%

%%%%%%%%%%%%%%%%%%%%%%%%%%%%%%%%%%%%%%%%%%%%%%%%%%%%%%%%%%%%%%%%%%%%%%

\section{Dwyer-Kan models for quasi-category mapping spaces}
\label{se:Dwyer-Kan}

In this section we use the Dwyer-Kan machinery of the previous section to give models for the
mapping spaces in a quasi-category.  These models have the advantage of being
relatively easy to work with and compute.  However, they have the disadvantage that
they do not admit a composition law.

\subsection{The canonical cosimplicial framing on $\sSetJ$}
\label{se:dwyer-kan1}
Recall from Section~\ref{se:quasi} that
$E\taking\Set\to\sSet$ denotes the 0-coskeleton functor.  For a set $S$, we
may also describe $ES$ as the nerve of the groupoid $E_GS$ with object
set $S$
and a single morphism $a\to b$ for each $a,b\in S$. 

Recall the Reedy model structure \cite[Chapter 15.3]{H} on the category of
cosimplicial objects in a model category.  In the present section, the
Reedy structure we will use will always be on $c(\sSet_J)$, not
$c(\sSet_K)$.  However, note that any Reedy-Joyal equivalence is also a
Reedy-Kan equivalence.

\begin{lemma}\label{lemma:ES fibrant qcat}
For any nonempty set $S$, the map $ES \ra \Delta^0$ is an acyclic fibration in
$\sSet_J$.
\end{lemma}  

\begin{proof}
The acyclic fibrations in the Joyal model category $\sSet_J$ are the same as
those in the Kan model category $\sSet_K$, as both model categories have the same cofibrations.  It
is easy to check that $ES\ra \Delta^0$ has the right lifting property
with respect to the maps $\bdd{n}\ra \del{n}$.
\end{proof}

The forgetful functor $\bD\inc \Set$ describes a cosimplicial set
whose $n$th object is $[n]=\{0,1,\ldots,n\}$.  Applying the functor
$E$ gives a cosimplicial object $[n]\mapsto E^n=N(E_G([n]))$ in
$\sSet_J$.  It is easy to check that $E^\bullet$ is Reedy cofibrant in
$c(\sSet_J)$, and the above lemma shows that each $E^n$ is
contractible in $\sSet_J$.  Note that the evident inclusion of
categories $[n]\inc E_G([n])$ induces a levelwise cofibration
$\Delta^\bullet \inc E^\bullet$, which in each level is a Kan
equivalence but not a Joyal equivalence.

For any simplicial set $X$, the cosimplicial object $[n]\mapsto
X\times E^n$ is a cosimplicial resolution of $X$ with respect to
the model structure $\sSet_J$.  

\subsection{Three cosimplicial versions of $\Delta^1$}\label{se:cosimplicial versions}

In his book, Lurie at various times uses three internal models for the
mapping space between vertices in a simplicial set $S$.  They are
called $\Hom^R_S$, $\Hom^L_S$, and $\Hom_S$; the descriptions of all
of these can be found in \cite[Section 1.2.2]{L}.  In this subsection
we show that  each of these can be understood as the mapping space
coming from a cosimplicial resolution of $\Delta^1$.  We also
give one new model, $\Hom^E_S$, which will be very useful later.
See Remark \ref{rem:our notation agrees with L} for more about the
Lurie models.

Recall that for any simplicial sets $M$ and $N$, the join $M\star N$
is a simplicial set with $$(M\star N)_n= \coprod_{-1\leq i\leq
n}M_i\times N_{n-i-1},$$ where we put $M_{-1}=N_{-1}=\Delta^0$ (see
\cite[1.2.8.1]{L}).  Note that $\star$ is a bifunctor and that $M\star\emptyset=\emptyset\star M=M$, so in particular there are natural inclusions
$M\inc M\star N$ and $N\inc M\star N$.  Note as well that
$M\join\Delta^0$ and $\Delta^0\join M$ are cones on $M$, and that $\Delta^n\join
\Delta^r\iso \Delta^{n+r+1}$.  

We let $C_R(M)$ and $C_L(M)$ denote the quotient $(M\star \Delta^0)/M$
and $(\Delta^0\star M)/M$, respectively.  For any simplicial set $M$,
let $\Ccyl(M)$ be the pushout
\[ \bdd{1} \la M\times \bdd{1} \inc  M\times \Delta^1.\]
Note that $C_R(M), C_L(M)$, and $\Ccyl(M)$ each has exactly
two $0$-simplices and comes with a natural map to $\Delta^1$ (which is unique over $\bd{1}$).

Let $C_R^\bullet$ (respectively $C_L^\bullet$) denote the cosimplicial
space $[n]\mapsto C_R(\Delta^n)$ (resp. $[n]\mapsto C_L(\Delta^n)$).
Write $\Ccyl^\bullet$ for the cosimplicial space $[n]\mapsto
\Ccyl(\Delta^n)$.  Finally, write $C_E^\bullet$ for the cosimplicial
space $[n]\mapsto \Ccyl(E^n)$.  

Note that there are inclusions  $C_R(\Delta^n)\inc \Ccyl(\Delta^n)$, and 
surjection maps $\Ccyl(\Delta^n)\ra C_R(\Delta^n)$ exhibiting $C_R(\Delta^n)$ as a
retract of $\Ccyl(\Delta^n)$.  
These assemble to give maps of cosimplicial spaces (in both
directions) between
$C_R^\bullet$ and 
$\Ccyl^\bullet$.  The same can be said if we replace $C_R$ with $C_L$. 
In addition, the map of cosimplicial spaces $\Delta^\bullet \ra E^\bullet$ gives us
a map $\Ccyl^\bullet\ra C_E^\bullet$.  In other words, we have maps
\begin{align}\label{dia:cosimplicial resolutions}\xymatrix@=12pt{ C_R^\bullet \ar[rrd]\\ && \Ccyl^\bullet \ar[rr]&& C_E^\bullet.\\C_L^\bullet\ar[rru]}
\end{align}
Note that in each of these cosimplicial spaces the $0$th space is
$\Delta^1$.  Hence, each of the above four cosimplicial spaces comes equipped
with a canonical map to $c\Delta^1$.  Each also comes equipped with a canonical map from $c(\bd{1})$. 

Recall the notation $\bpsSet=(\bdd{1}\ovcat\sSet_J)$, and let $c(\bpsSet)$ denote the Reedy model structure on the cosimplicial objects in $\bpsSet$, as in Section \ref{se:mapping-spaces}.

\begin{prop}\mbox{}\par
\label{pr:three-resolutions}
\begin{enumerate}[(a)]
\item Each of the maps $c(\bdd{1})\ra C_R^\bullet$,  $c(\bdd{1})\ra C_L^\bullet$,
$c(\bdd{1})\ra \Ccyl^\bullet$,  and $c(\bdd{1})\ra C_E^\bullet$ is a Reedy cofibration.
\item Each of $C_R^\bullet$, $C_L^\bullet$, $\Ccyl^\bullet$, and
$C_E^\bullet$ is Reedy cofibrant as an object of $c(\bpsSet)$.
\item Each of the maps of simplicial sets 
$C_R^n \ra \Delta^1$, $C_L^n \ra \Delta^1$, $\Ccyl^n\ra \Delta^1$, and
$C_E^n \ra \Delta^1$ is a Joyal equivalence.
\item 
Consequently, each
of $C_R^\bullet$, $C_L^\bullet$, $\Ccyl^\bullet$, and $C_E^\bullet$ is a cosimplicial
resolution of $\Delta^1$ in $c(\bpsSet)$.    
\end{enumerate}
\end{prop}

\begin{proof}
Parts (a) and (b) are obvious, and (d) follows immediately from (b)
and (c).  For (c), note first that $C^n_E\ra \Delta^1$ is easily seen
to be a Joyal equivalence.  Indeed, $C^n_E$ is the pushout of
\[ \bd{1} \llla{\sim} \bd{1}\times E^n \inc \Delta^1\times E^n,
\]
where the indicated map is a Joyal equivalence by Lemma~\ref{lemma:ES
fibrant qcat}.  It follows from left properness of $\sSet_J$ that
$\Delta^1 \times E^n \ra C^n_E$ is a Joyal equivalence.  Using that
$\Delta^1\times E^n \ra \Delta^1$ is a Joyal equivalence
(Lemma~\ref{lemma:ES fibrant qcat} again), it follows immediately that
$C^n_E \ra \Delta^1$ is also one.

The arguments for $C^n_R$, $C^n_L$, and $\Ccyl^n$ are more
complicated.  Picking the former for concreteness, there are various
sections of the map $C^n_R \ra \Delta^1$.  It will be sufficient to
show that any one of these is a Joyal acyclic cofibration, which we do by
exhibiting it as a composition of cobase changes of inner horn
inclusions.  This is not difficult, but it is a little cumbersome; we
postpone the proof until Section~\ref{se:leftover} (see Proposition~\ref{prop:squash}).
\end{proof}

\subsection{Application to mapping spaces}
For every $S\in \sSet_J$ and every $a,b\in S$,
define $\Hom^R_S(a,b)$ to be the simplicial set
$\bpsSet(C^\bullet_R,S)$.  Note that this is also the pullback of
\[ * \llra{(a,b)} S\times S \bfib \sSet(C_R^\bullet,S).
\]  
Define $\Hom^L_S(a,b)$, $\Hom^{\cyl}_S(a,b)$, and $\Hom^E_S(a,b)$ analogously, and note
that diagram (\ref{dia:cosimplicial resolutions}) induces natural maps
\begin{myequation}
\label{dia:comparing homs}
\xymatrix@=12pt{&&&&\Hom^R_S(a,b)\\\Hom^E_S(a,b) \ar[rr]&&
\Hom^{\cyl}_S(a,b) \ar[rru]\ar[rrd]\\ &&&&\Hom^L_S(a,b).}
\end{myequation}

\begin{cor}\label{cor:four models}
When $S\in \sSet_J$ is fibrant and $a,b\in S$, the four natural maps
in (\ref{dia:comparing homs}) are Kan equivalences of simplicial
sets.  These simplicial sets are models for the homotopy function
complex $\hMap_{\bpsSet}(\Delta^1,S)$, where $S$ is regarded as an
object of $\bpsSet$ via the map $\bdd{1}\ra S$ sending $0\ra
a$ and $1\ra b$.
\end{cor}

\begin{proof}
This is immediate from 
Proposition~\ref{pr:three-resolutions}
and \cite[16.5.5]{H}.  
\end{proof}

\begin{remark}\label{rem:our notation agrees with L}

For a simplicial set $S$ and vertices $a,b\in S_0$, our notation
$\Hom^R_S (a,b)$ and $\Hom^L_S(a,b)$ agrees with that of \cite[Section
1.2.2]{L}.  Our notation 
$\Hom^{\cyl}_S(a,b)$ is denoted $\Hom_S(a,b)$ in \cite{L}, and we used
the Lurie notation earlier in Proposition~\ref{prop:observation}.
Note that
$\Hom_S^{cyl}(a,b)$ can also be described as the fiber of the morphism
of simplicial mapping spaces
$\Map_\sSet(\Delta^1,S)\to\Map_\sSet(\bd{1},S)$ over the point
$(a,b)$.  The model $\Hom^E_S(a,b)$ does not seem to appear in
\cite{L}, but will be very useful in Section~\ref{se:relative}.

\end{remark} 

The following calculation will be needed in the next section.  Recall
the notion of {\it necklace\/}, from \cite[Section 3]{DS1}, and that
if $T=\Delta^{n_1}\Wedge\cdots\Wedge\Delta^{n_k}$ is a necklace then
$\Delta[T]$ denotes the simplex spanned by the ordered set of vertices
of $T$.

\begin{prop}
\label{pr:necklace,hmap}
Let $T$ be a necklace.  Then $\hMap_{\bpsSet}(\Delta^1,T)$ is contractible.
\end{prop}

\begin{proof}
It is a fact that $T\ra \Delta[T]$ is a
Joyal equivalence; see
Lemma~\ref{le:necklace=simplex} for a proof.
Also, $\Delta[T]$ is fibrant in
$\sSet_J$ because it is the nerve of a category (as is any
$\Delta^k$).  We may therefore model our homotopy function complex by
\[ \bpsSet(C^\bullet_R,\Delta[T])
\]
where $C^\bullet_R$ is the cosimplicial resolution of $\Delta^1$
considered in this section.  

It is easy to check that in $\bpsSet$ there is a unique map from
$C^n_R$ to $\Delta[T]$, for each $n$ (it factors through $\Delta^1$).  Therefore we have
$\bpsSet(C^\bullet_R,\Delta[T])=*$, and this completes the proof.
\end{proof}

%%%%%%%%%%%%%%%%%%%%%%%%%%%%%%%%%%%%%%%%%%%%%%%%%%%%%%%%%%%%%%%%%%%%%%

\section{Connections with the rigidification mapping spaces}
\label{se:connect}

In this section we prove that for any simplicial set $S$ and any
$a,b\in S_0$, the categorification mapping space $\jC(S)(a,b)$
is naturally Kan equivalent to the Dwyer-Kan mapping space
$\hMap_{(\bpsSet)_J}(\Delta^1,S_{a,b})$.  As a corollary, we prove
that for any simplicial category $\cD$ the counit map $\jC(N\cD)\ra \cD$ is a
weak equivalence in $\sCat$.

\medskip

If $Y$ is an object in $\bpsSet$ we will write
$\alpha$ and $\omega$ for the images of $0$ and $1$ under the map
$\bd{1}\ra Y$.  We always use the Joyal model structure on the
category $\bpsSet=(\bd{1}\ovcat\sSet)$.

If $S$ is a simplicial set and $a,b\in S$, let us use the notation
$\hMap(S)(a,b)$ as shorthand for a homotopy function complex
$\hMap_{\bpsSet}(\Delta^1,S_{a,b})$.

Let $\cY$ denote the full subcategory of $\bpsSet$ whose objects are
spaces $Y$ such that $\hMap(Y)(\alpha,\omega)\he *$ and
$\jC(Y)(\alpha,\omega)\he *$.  Note that $\cY$ contains the category
$\Nec$ (the category of necklaces) by Proposition~\ref{pr:necklace,hmap} and
\cite[Corollary 3.8]{DS1}.
So clearly $\cY$ is a {\it category of gadgets\/} in the
sense of \cite[Definition 5.4]{DS1}.  Let $\cY_f$ denote the full
subcategory of $\cY$ consisting of those objects which are fibrant in
$\sSet_J$.  Let $\jC^\cY$ and $\jC^{\cY_f}$ be the corresponding
functors $\sSet\to\sCat$, as defined in \cite[Section 5.3]{DS1}.  So
if $S\in \sSet$ and $a,b\in S_0$, $\jC^\cY(S)(a,b)$ is the nerve of
the category $(\cY\ovcat S)_{a,b}$, and similarly for $\jC^{\cY_f}(S)(a,b)$.

\begin{remark}
The two conditions that define $\cY$ seem like they should be equivalent, and they are.  That is, we will show in Corollary \ref{co:connect1} that the conditions $\jC(Y)(\alpha,\omega)\he *$ and
$\hMap(Y)(\alpha,\omega)\he *$ are equivalent.  However, at the moment we do not know this, so including both conditions in the
definition of $\cY$ is not redundant.
\end{remark}

Let $C^\bullet$ be a cosimplicial resolution of $\Delta^1$ in
$\sSet_J$.  Let $C^\bullet\trcof R^\bullet \trfib c(\Delta^1)$ be a
factorization into a Reedy acyclic cofibration followed by Reedy
fibration (which will necessarily be acyclic as well).    By
\cite[Prop. 15.3.11]{H} the maps $R^n\ra \Delta^1$ are Joyal fibrations.  Since $\Delta^1$ is Joyal fibrant (being the nerve
of a category), the objects $R^n$ are fibrant as well.

\begin{prop}
\label{pr:connect1}
If $S$ is fibrant in $\sSet_J$ and $a,b\in S_0$, then
there is a natural commutative diagram in which all the maps are Kan
equivalences:
\[\xymatrix{
\jCn(S)(a,b) \ar[r]^\sim & \jC^{\cY}(S)(a,b) & \jC^{\cY_f}(S)(a,b) \ar[l]_\sim
\\
& N\bD \bpsSet(C^\bullet,S_{a,b})\ar[u]^\sim & N\bD\bpsSet(R^\bullet,S_{a,b}).\ar[l]^\sim\ar[u]_\sim
}
\]
\end{prop}

\begin{proof}
The map $\jCn(S)(a,b)\ra
\jC^{\cY}(S)(a,b)$ is induced by the inclusion of categories $\Nec\inc
\cY$, and 
\cite[Proposition 5.5]{DS1} shows that is  a Kan equivalence.
The map $\jC^{\cY_f}(S)(a,b) \ra \jC^{\cY}(S)(a,b)$ is the
 nerve of the evident
inclusion of categories $j\colon (\cY_f\ovcat S)_{a,b} \ra (\cY\ovcat
S)_{a,b}$.  Let us show it is a Kan equivalence.  For a simplicial set $X$, let $X\trcof\hat{X}$ denote
a (functorial) fibrant replacement of $X$ in $\sSet_J$.  Since $S$ is
fibrant, there is a map $\hat{S}\ra S$ such that the composition $S\ra
\hat{S}\ra S$ is the identity.  Define a functor
\[ F\colon (\cY\ovcat S)_{a,b} \ra (\cY_f\ovcat S)_{a,b} \]
by sending the pair $[Y,Y\ra S]$ to the pair $[\hat{Y},\hat{Y}\ra
\hat{S}\ra S]$.  For this to make sense we need to know that $\hat{Y}$
is in $\cY_f$; this is true because changing from $Y\ra \hat{Y}$ does
not change the Dwyer-Kan mapping space $\hMap(\blank)(a,b)$ nor, by
\cite[Proposition 6.6]{DS1}, the $\jC(\blank)(a,b)$ mapping
space.  It is easy to see that there is a natural transformation
between the composite $jF$ (resp. $Fj$) and the identity, so $j$
induces a Kan equivalence of the nerves.

Next consider the map $N\bD\bpsSet(R^\bullet,S_{a,b}) \ra
\jC^{\cY_f}(S)(a,b)$.  This is again the nerve of a functor
\[ f\colon \bD\bpsSet(R^\bullet,S_{a,b}) \ra (\cY_f\ovcat S)_{a,b} 
\]
which sends $[[n],R^n\ra S]$ to $[R^n,R^n\ra S]$.  We will verify that
the overcategories of $f$ are contractible, hence it induces a Kan
equivalence of the nerves by Quillen's Theorem A.  (For typographical
reasons, we will drop the subscripts $a,b$, etc.)  Pick an object
$y=[Y,Y\ra S]$ in $(\cY_f\ovcat S)$.  The overcategory $(f\ovcat y)$
has objects $[[n],R^n\ra Y]$ and the evident morphisms; that is,
$(f\ovcat y)=\bD\bpsSet(R^\bullet,Y)$.  But since $Y$ is Joyal fibrant,
$\bpsSet(R^\bullet,Y)$ is a model for $\hMap(Y)(a,b)$, and this is
contractible because $Y\in \cY_f$.

The map $\bpsSet(R^\bullet,S)\ra \bpsSet(C^\bullet,S)$ is a Kan equivalence
because $C^\bullet \ra R^\bullet$ is a Reedy weak equivalence between
Reedy cofibrant objects and $S$ is Joyal fibrant; see
\cite[16.5.5]{H}.  Hence, the map $N\bD\bpsSet(R^\bullet,S_{a,b})\ra
N\bD\bpsSet(C^\bullet,S_{a,b})$ is a Kan equivalence, using the fact
that $N\Delta K\he K$, for any simplicial set $K$ (see
\cite[Theorem 18.9.3]{H}).

The final map $\bD\bpsSet(C^\bullet,S_{a,b})\ra \jC^{\cY}(S)(a,b)$ is a Kan
equivalence by the two-out-of-three property.
\end{proof}

\begin{cor}
\label{co:connect1}
Let $C^\bullet$ be any cosimplicial resolution for $\Delta^1$ in
$\bpsSet$.  For a quasi-category $S$ and $a,b\in S_0$, there is a
natural zig-zag of Kan equivalences between $\jC(S)(a,b)$ and
$\hMap(S)(a,b)=\bpsSet(C^\bullet,S_{a,b})$.  
\end{cor}

\begin{proof}
Recall from \cite[Theorem 5.2]{DS1} that there is a natural zig-zag of Kan
equivalences between $\jCn(S)(a,b)$ and $\jC(S)(a,b)$.  Also recall
that for any simplicial set $K$, there is a natural zig-zag of Kan
equivalences between $K$ and $N\Delta K$ by \cite[Theorem 18.9.3]{H}.  The corollary
follows immediately from combining these facts with
Proposition~\ref{pr:connect1}.
\end{proof}

For the rest of this section we write $\cA=\bpsSet$, to ease the
 typography.

Proposition~\ref{pr:connect1} gives a simple zig-zag of Kan equivalences
between $\jCn(S)(a,b)$ and $N\bD\bpsSet(C^\bullet,S_{a,b})$ for any
cosimplicial resolution $C^\bullet$ of $\Delta^1$ in $\sSet_J$.  In Proposition \ref{pr:connect2} we will present another simple zig-zag which is sometimes useful.  Define
\begin{align}\label{dia:phi} \phi\colon (\Nec\ovcat S)_{a,b} \ra \cA W^{-1}\cA(\Delta^1,S_{a,b})
\end{align}
by sending $[T,T\ra S]$ to $[\Delta^1 \ra \Delta[T] \bwe T \ra S]$.  
Here $\Delta[T]$ is the associated simplex to $T$, described in
\cite[Section 3]{DS1}, which is functorial in $T$.  The map
$\Delta^1 \ra \Delta[T]$ is the unique $1$-simplex connecting the
initial and final vertices.  Note that there is also a functor
\begin{myequation}
\label{dia:j} j\colon \bD\bpsSet(C^\bullet,S)\ra \cA W^{-1}\cA
(\Delta^1,S)
\end{myequation}
which sends $[[n],C^n\ra S]$ to $[\Delta^1 \llra{\id}
\Delta^1 \bwe C^n \ra S]$, and by Remark~\ref{re:DK-summary}
this functor induces a Kan equivalence on nerves.

\begin{prop}
\label{pr:connect2}
For any fibrant simplicial set $S\in \sSet_J$ and $a,b\in S_0$, the
maps
\[ \jCn(S)(a,b) \llra{N\phi} N \bigl [\cA W^{-1}\cA(\Delta^1,S_{a,b})\bigr ] \llla{Nj}
N\bigl [\bD \bpsSet(C^\bullet,S_{a,b})\bigr ],
\] 
are Kan equivalences,
where $\phi$ and $j$ are as in (\ref{dia:phi}) and (\ref{dia:j}\!).
\end{prop}

\begin{proof}
We consider the following diagram of categories, where we have
suppressed all mention of $a$ and $b$, but everything is suitably over
$\bdd{1}$:
\[ \xymatrixcolsep{1.2pc}\xymatrix{
(\Nec\ovcat S) \ar[d]_\phi\ar[rr] && (\cY\ovcat S) &
\bD\bpsSet(C^\bullet,S).\ar[dl]^{j_0}\ar[l]
\\
\cA W^{-1}\cA(\Delta^1,S) \ar@<0.5ex>[r]^-{\pi_2} & \cA (\Wfib)^{-1}\cA(\Delta^1,S)
\ar@<0.5ex>[r]^-{\pi_1}\ar@<0.5ex>[l]^-{j_2} & \cA W^{-1}(\Delta^1,S) \ar[u]_i\ar@<0.5ex>[l]^-{j_1} }
\]  

The nerve of each map in the top row is from Proposition~\ref{pr:connect1}, where
it is shown to be a Kan equivalence.  The map $\phi$ was defined above.  The maps $j_0,j_1,j_2,i,\pi_1,$ and $\pi_2$ are in some sense self-evident, but we describe them now (in that order).  The symbol ``\; $\sim$\;\;" in this proof always denotes a Joyal equivalence.

The map $j_0$ sends $[[n],C^n\ra S]$ to $[\Delta^1
\bwe C^n \ra S]$; $j_1$ sends $[\Delta^1 \bwe X \ra S]$ to $[\Delta^1
\llra{\id}\Delta^1 \btrfib X' \ra S]$, where $X\trcof X'\fib\Delta^1\cross S$ is a functorial factorization of $X\to\Delta^1\cross S$; and $j_2$ is induced by the
inclusion $\Wfib\inc W$.  Note that there is a natural transformation
$j\ra j_2j_1j_0$, so these maps induce homotopic maps on nerves.  The map $i$ sends
$[\Delta^1 \bwe X \ra S]$ to the pair $[X,X\ra S]$ (note that if $X\he
\Delta^1$ then $X\in \cY$ by \cite[Proposition 6.6]{DS1} and the
homotopy invariance of the Dwyer-Kan mapping spaces).  Finally, the
maps $\pi_1$ and $\pi_2$ are functors giving homotopy  inverses to
$j_1$ and $j_2$.  The functor $\pi_1$ sends $[\Delta^1 \ra X \btrfib Y
\ra S]$ to $[\Delta^1 \bwe (\Delta^1\times_X Y) \ra S]$, and $\pi_2$
sends the zig-zag $[\Delta^1 \ra X \bwe Y \ra S]$ to $[\Delta^1 \ra X \btrfib Y' \ra
S]$ where $Y'$ is obtained from the functorial factorization of $Y\ra
X\times S$ into $Y\trcof Y' \fib X\times S$.  It is easy to see that
there are natural transformations between the composite $j_i\pi_i$,
$\pi_ij_i$, and their respective identities, thus showing that these
maps are homotopy inverses.

Next one should check that the functor $i\pi_1\pi_2\phi$ is connected
to the top map $(\Nec\ovcat S)\ra (\cY\ovcat S)$ by a zig-zag of natural
transformations (this is easy), and hence the two maps induce
homotopic maps on nerves.  So the (nerve of the) large rectangle in the above diagram
commutes in the homotopy category.  The right-hand triangle commutes on the nose.

The map $j_0$ induces a Kan equivalence on nerves by
Remark~\ref{re:DK-summary}.  Returning to our original diagram and the
sentence immediately following it, the two-out-of-three property implies
that $i$ induces a Kan equivalence on nerves.  We have already shown
that $\pi_1\pi_2$ and $j_2j_1$ do so as well; therefore the same is
true for $\phi$ and $j$.
\end{proof}

\begin{remark}
The above result in some sense explains why necklaces might arise in
models for mapping spaces, as they did in \cite{DS1}.  If $T$ is a
necklace then a map $T\ra S_{a,b}$ gives us, in a canonical way, a
zig-zag
\[ \Delta^1 \inc \Delta[T] \bwe T \ra S\] in $\cA W^{-1}\cA(\Delta^1,S)$, 
which represents a map $\Delta^1 \ra S$ in $\ho(\sSet_J)$.  
\end{remark}

\subsection{The counit of categorification}
Our next result concerns the counit $\epsilon\taking \jC
N\to\id_\sCat$ for the adjunction $\jC\colon \sSet_J \adjoint
\sCat\colon N$.  The proof is only a slight modification of
that for Proposition~\ref{pr:connect1} above.  For a proof using very
different methods, see \cite[Theorem 2.2.0.1]{L}.

\begin{prop}
\label{pr:counit}
Let $\cD$ be a simplicial category all of whose mapping spaces are Kan
complexes.  Then the counit map $\jC N\cD\to\cD$ is a weak equivalence
in $\sCat$.
\end{prop}

\begin{proof}

Since $\jC(N\cD)$ is a simplicial category with the same object set as
$\cD$, it suffices to show that for any $a,b\in\ob\cD$ the map \[
\jC(N\cD)(a,b) \ra \cD(a,b)\] is a Kan equivalence.

Let $C^\bullet$ be the cosimplicial resolution $C^\bullet_R$ from
Section~\ref{se:Dwyer-Kan}, so that we have $C^n=(\Delta^n \star
\Delta^0)/\Delta^n$.  Observe that $\jC(C^n)$ is a simplicial category
with two objects $0$ and $1$, and following \cite[Section 2.2.2]{L}
let $Q^n$ denote the mapping space $\jC(C^n)(0,1).$  By
\cite[Proposition 6.6]{DS1} and
Proposition~\ref{pr:three-resolutions}(c) the map $Q_n\ra
\jC(\Delta^1)(0,1)=*$ is a Kan equivalence, hence $Q^n$ is
contractible.  Also, since $C^\bullet_R$ is Reedy cofibrant it follows
readily that $Q^\bullet$ is also Reedy cofibrant.  So
the cosimplicial
space $Q^\bullet$ is a Reedy cosimplicial resolution of a point in
$\sSet_K$. 

Consider the following diagram in $\sSet_K$:
\[ \xymatrix{
\jC(N\cD)(a,b) \ar[r] & \cD(a,b) \\
\colim\limits_{T\ra N\cD} \jC(T)(\alpha,\omega) \ar[u]\ar[r]&
\colim\limits_{Y\ra N\cD} \jC(Y)(\alpha,\omega) \ar[u]
& \colim\limits_{[n],C^n\ra N\cD} \jC(C^n)(\alpha,\omega)\ar[l] \\
\hocolim\limits_{T\ra N\cD} \jC(T)(\alpha,\omega) \ar[u]\ar[r]\ar[d]_\sim &
\hocolim\limits_{Y\ra N\cD} \jC(Y)(\alpha,\omega) \ar[d]_\sim\ar[u]
& \hocolim\limits_{[n],C^n\ra N\cD} \jC(C^n)(\alpha,\omega)\ar[d]_\sim\ar[l]\ar[u] \\
\hocolim\limits_{T\ra N\cD} \,{*} \ar[r] & \hocolim\limits_{Y\ra N\cD} \,{*} &
\hocolim\limits_{[n],C^n\ra N\cD} {*}\ar[l]
}
\]
For the colimits in the left-hand column the indexing category is
$(\Nec\ovcat N\cD)_{a,b}$.  For the middle column it is $(\cY\ovcat
N\cD)_{a,b}$, where $\cY$ is the category of gadgets described at the beginning
of this section.  For the right-hand column the colimits are indexed
by the category $\bD\bpsSet(C^\bullet,N\cD_{a,b})$.  The maps between
columns (except at the very top) come from the evident maps between
indexing categories.  Finally, the top vertical map in the middle
column comes from taking a map $Y\ra N\cD$, adjointing it to give
$\jC(Y)\ra \cD$, and then using the induced map
$\jC(Y)(\alpha,\omega)\ra \cD(a,b)$.  It is easy to see that the
diagram commutes.

The indicated maps are Kan equivalences because the mapping spaces in
$\jC(T)$, $\jC(Y)$ and $\jC(C^n)$ are all contractible.  The bottom
horizontal row is $\jCn(N\cD)(a,b) \ra \jC^\cY(N\cD)(a,b) \la
N\bD\bpsSet(C^\bullet,N\cD_{a,b})$, and these maps are Kan equivalences by
Proposition~\ref{pr:connect1}.  It follows that the horizontal maps in
the third row are all Kan equivalences as well. 

Now, the map $\hocolim_{[n],C^n\ra N\cD} \jC(C^n)(\alpha,\omega) \ra
\cD(a,b)$ can be written as
\[ \hocolim_{[n],\jC(C^n)\ra \cD} \jC(C^n)(\alpha,\omega) \ra 
\colim_{[n],\jC(C^n)\ra \cD} \jC(C^n)(\alpha,\omega) \ra \cD(a,b).
\]
To give a map $\jC(C^n)\ra \cD$ over $a,b$ is exactly the same as giving a map
$Q^n=\jC(C^n)(\alpha,\omega) \ra \cD(a,b)$.  So the above maps may
also be written as
\[ \hocolim_{[n],Q^n\ra \cD(a,b)} Q^n \lra
\colim_{[n],Q^n\ra \cD(a,b)} Q^n \lra \cD(a,b).
\]
By Lemma~\ref{le:hoco} below (using that $\cD(a,b)$ is a Kan complex),
this composite is a Kan equivalence.

It now follows from our big 
diagram that $\hocolim_{Y\ra N\cD} \jC(Y)(\alpha,\omega) \ra \cD(a,b)$
is a Kan equivalence.  Finally, by \cite[Theorem 5.2]{DS1}
the map 
\[ \hocolim_{T\ra N\cD} \jC(T)(\alpha,\omega) \ra
\jC(N\cD)(a,b)
\] 
is a Kan equivalence (this is the map $\jCh(N\cD)(a,b)
\ra \jC(N\cD)(a,b)$ from the statement of that theorem). 
 It now follows at once that $\jC(N\cD)(a,b) \ra
\cD(a,b)$ is a Kan equivalence.
\end{proof}

\begin{lemma}
\label{le:hoco}
Let $U^\bullet$ be any cosimplicial resolution of a point with respect
to $\sSet_K$.  Then
for any Kan complex $X$, the composite
\[ \hocolim_{[n],U^n\ra X} U^n \lra
\colim_{[n],U^n\ra X} U^n \lra X.
\]
is a Kan equivalence. 
\end{lemma}

\begin{proof}
The result is true for the cosimplicial resolution $\Delta^\bullet$ by
a standard result; see \cite[Prop. 19.4]{D2}, for instance.  There is a zig-zag
$U^\bullet \ra V^\bullet \la \Delta^\bullet$ of Reedy weak
equivalences, where $V^\bullet$ is a cofibrant-fibrant replacement of
$\Delta^\bullet$ in $c(\sSet_K)$.   Because of this it is sufficient to show that if
$U^\bullet \ra V^\bullet$ is a map between cosimplicial resolutions of
a point and we know the result for one of them, then we also know it
for the other.

Let $I=\Delta\cM(U^\bullet,X)$ and
$J=\Delta\cM(V^\bullet,X)$, and observe that our map
$U^\bullet\ra V^\bullet$ induces a
functor $f\colon J\ra I$.  

Let $\Gamma_U\colon I\ra \cM$ be the functor $[[n],U^n\ra X]\mapsto U^n$
and let $\Gamma_V\colon J\ra \cM$ be the functor $[[n],V^n\to X]\mapsto V^n$.  Finally, let $\Theta\colon J\ra \cM$
be the functor $[[n],V^n\mapsto X]\mapsto U^n$.  Note that there
is a natural transformation $\Theta\ra \Gamma_V$, and also that 
$\Theta=\Gamma_U\circ f$.

One considers the following diagram:
\[ \xymatrix{
N\Delta\cM(U^\bullet,X) \ar@{=}[r] & \hocolim_I {*} 
& \hocolim_{I} \Gamma_U \ar[l]_\sim\ar[r] & \colim_I \Gamma_U
\ar[dr] \\
N\Delta\cM(V^\bullet,X)\ar@{=}[r]\ar[u] & \hocolim_J {*} \ar[u] 
& \hocolim_J \Theta \ar[u]\ar[l]_\sim\ar[d]^\sim\ar[r] & \colim_J \Theta \ar[u]
\ar[d]\ar[r] & X. \\
&& \hocolim_J \Gamma_V \ar[r] & \colim_J \Gamma_V \ar[ur]
}
\]
The maps labelled $\sim$ are Kan equivalences because all the values
of $\Gamma_V$, $\Gamma_U$, and $\Theta$ are contractible.  

The key observation is that the
map $\cM(V^\bullet,X)\ra \cM(U^\bullet,X)$ is a Kan equivalence
by \cite[16.5.5]{H}, since both $V^\bullet$ and $U^\bullet$ are
cosimplicial resolutions of a point in $\sSet_K$ and $X$ is Kan fibrant.
It follows that $N\Delta\cM(V^\bullet,X)\ra
N\Delta\cM(U^\bullet,X)$ is also a Kan equivalence, and applying the
two-out-of-three axiom to the diagram we obtain that
$\hocolim_I\Gamma_U \ra X$ is a Kan equivalence if and only if
$\hocolim_J \Gamma_V\ra X$ is a Kan equivalence.  This is what we wanted.
\end{proof}

%%%%%%%%%%%%%%%%%%%%%%%%%%%%%%%%%%%%%%%%%%%%%%%%%%%%%%%%%%%%%%%%%%%%%%

\section{Relative mapping spaces}
\label{se:relative}

In previous sections we studied the mapping spaces
$\hMap_{(\bd{1}\ovcat \sSet_J)}(\Delta^1,X)$.  There is an evident
generalizaton of this construction which replaces $\bd{1}\cof\del{1}$
with an arbitrary cofibration $A\cof B$.  This turns out to be very
useful, and the purpose of the present section is to develop the basic
properties of these relative mapping spaces.  

\medskip

Recall from  Lemma~\ref{lemma:ES fibrant qcat} 
that the map $A\cross E^n\to A$ is a Joyal acyclic fibration, for all
$A$.  It follows that the cosimplicial object $[n]\mapsto A\times E^n$ is a
cosimplicial resolution for $A$ with respect to $\sSet_J$.

\begin{defn}
Fix a cofibration $A\cof B$.
\begin{enumerate}[(a)]
\item
Define $C_E^\bullet(B,A)$ to be the cosimplicial object obtained as
the pushout
\[ \xymatrix{cA\times E^\bullet \ar@{ >->}[r]\ar[d]_\sim & cB\times E^\bullet \ar@{.>}[d]
\\
 cA \ar@{ >.>}[r] & C_E^\bullet(B,A).
}
\]
Note that the pushout of a Joyal equivalence is still a Joyal
equivalence by left properness, and so $C_E^\bullet(B,A)$ is a
cosimplicial resolution of $B$ in the model category
$(A\ovcat \sSet_J)$.  
\item For any quasi-category $X$ and any fixed map $f\colon A\ra X$,
let $\hMap_A(B,X)$ be the pullback of simplicial sets
\[ \xymatrix{
\hMap_A(B,X) \ar@{.>}[r]\ar@{.>>}[d] & \sSet(C_E^\bullet(B,A),X)\ar@{->>}[d] \\
   {*} \ar[r]_f & \sSet(cA,X).
}
\]
Note that $\sSet(cA,X)$ is a discrete simplicial set, and also that
the right vertical map is a Kan fibration since $cA\ra
C_E^\bullet(B,A)$ is a Reedy cofibration.  One should also observe
that $\hMap_A(B,X)$ depends on the map $f$ as well as on the fixed cofibration, although this is obscured
in the notation.  
\end{enumerate}
\end{defn}

\begin{remark}
\label{re:rel}
The simplicial set $\hMap_A(B,X)$ is simply a particular model for the
homotopy function complex from $B$ to $X$ in the model category
$(A\ovcat \sSet_J)$.  Note that of the four resolutions considered
in Proposition~\ref{pr:three-resolutions}, only the $E^\bullet$
resolution is relevant in our present context; the others 
are specific to $\bd{1}\to\Delta^1$.

For later use, observe that $\hMap_A(B,X)$ can also be described as the pullback
\[  * \lra \sSet(A\times E^\bullet,X) \lla \sSet(B\times E^\bullet,X)
\]
where the left map is the composite $* \ra \sSet(cA,X) \ra
\sSet(A\times E^\bullet,X)$.  
\end{remark}

\begin{prop} Let $A\cof B$ be a cofibration, and let $C$ be any
simplicial set.  Let $X$ be a quasi-category and $f\colon A\ra X^C$ be
a map.  Then
\[ \hMap_A(B,X^C)\iso \hMap_{A\times C}(B\times C,X)
\]
where the right mapping space is relative to the map $A\times C\ra X$
adjoint to $f$.
\end{prop}

\begin{proof}
Easy, using the second statement in Remark~\ref{re:rel}.
\end{proof}

Given a square
\begin{myequation}
\label{eq:square}
 \xymatrix{ A \ar@{ >->}[r]\ar[d] & B\ar[d] \\
 A' \ar@{ >->}[r] & B'}
\end{myequation}
and a map $A'\ra X$,
there is an induced map $\hMap_{A'}(B',X) \ra \hMap_A(B,X)$.  
The next two results give properties of these natural maps.

\begin{prop}
\label{pr:rel-fib}
Assume given a square such as (\ref{eq:square}), and let $L\colon
A'\amalg_A B\ra B'$ denote the induced ``latching'' map.  Then for any
quasi-category $X$ and any map $A'\To{f} X$, the induced map
$\hMap_{A'}(B',X) \ra \hMap_A(B,X)$ is a Kan fibration if $L$ is a
cofibration, and it is a Kan acyclic fibration if $L$ is a Joyal acyclic
cofibration.
\end{prop}

\nosee{%2009/10/09

\begin{lemma}
\label{le:fibers}
Let
\[ \xymatrix{
 X\ar[d]_f\ar[r]  & Y\ar[d]^g \\
 Z\ar[r] & W
}
\]
be a diagram of simplicial sets, and let $z\in Z$ be some basepoint.
Let $F_1$ be the fiber of $f$ over $z$, and let $F_2$ be the fiber of
$g$ over the image of $z$ .  Then if the ``matching map'' $X\ra Y\times_Z W$ has
the right-lifting-property with respect to some map $A\ra B$, so does
the map of fibers $F_1\ra F_2$.

In particular, if the matching map is a Kan fibration (resp. acyclic
fibration) so is $F_1\ra F_2$.  
\end{lemma}

\begin{proof}
Left to the reader.
\end{proof}

}%2009/10/09

\begin{proof}
Consider the diagram below:
$$\xymatrix{\hMap_{A'}(B',X)\ar[r]\ar[d]&\sSet(B'\cross
E^\bullet,X)\ar[d]^{(L\cross
\id)^*}\\\hMap_A(B,X)\ar[r]\ar[d]&\sSet((A'\amalg_AB)\cross
E^\bullet,X)\ar[r]\ar[d]&\sSet(B\cross E^\bullet,X)\ar[d]\\
{*}\ar[r]_-f&\sSet(A'\cross E^\bullet,X)\ar[r]&\sSet(A\cross
E^\bullet,X).}
$$
Note that the two large rectangles are pullback squares, and that the
lower right square is also a pullback.  It follows by category
theory that the lower left square is also a pullback, and then that
the same is true for the upper left square.
The desired result now follows directly, since the map labelled $(L\times
\id)^*$ will be either a Kan fibration or acyclic Kan fibration under
the respective hypothesis on $L$.
\end{proof}

\begin{prop}
\label{pr:relative-pullback}
Let $A\cof B$ be a cofibration, and let $B_1\cof B$ and $B_2\cof B$ be
such that $B_1\cup B_2=B$.  Let $X$ be a quasi-category and $f\colon A\ra X$ any map.  
Then the following is a pullback square, as well as a homotopy
pullback square in $\sSet_K$:
\[ \xymatrix{
\hMap_A(B,X) \ar@{->>}[r]\ar@{->>}[d] & \hMap_{B_1\cap A}(B_1,X)\ar@{->>}[d] \\
\hMap_{B_2\cap A}(B_2,X) \ar@{->>}[r] & \hMap_{B_1\cap B_2\cap A}(B_1\cap
B_2,X).
}
\]
\end{prop}

\begin{proof}
Note that the pullback statement immediately implies the homotopy
pullback statement, since we know by four applications of
Proposition~\ref{pr:rel-fib} that the indicated maps are Kan
fibrations.

To prove the pullback statement 
we simply consider the following diagram:
\[ \xymatrix{
{*}\ar[r]\ar[d] & \sSet((B_1\cap A)\times E^\bullet,X) \ar[d]  &
\sSet(B_1\times E^\bullet,X) \ar[d]\ar[l] \\
{*}\ar[r] & \sSet((B_1\cap B_2\cap A)\times E^\bullet,X)   &
\sSet((B_1\cap B_2)\times E^\bullet,X) \ar[l] \\
{*}\ar[r]\ar[u] & \sSet((B_2\cap A)\times E^\bullet,X) \ar[u]  &
\sSet(B_2\times E^\bullet,X). \ar[u]\ar[l] 
}
\]
The limit of this diagram may be constructed by first forming the
pullbacks of the columns, and then forming the resulting pullback; and
it may also be constructed by first forming the pullbacks of the {\it rows\/}
and then forming the resulting pullback.  
The former method gives the pullback of
\[ * \ra \sSet(A\times E^\bullet,X) \la \sSet(B\times E^\bullet,X), \]
which is just $\hMap_A(B,X)$.   The latter method gives the pullback of
\[ \hMap_{B_1\cap A}(B_1,X) \ra \hMap_{B_1\cap B_2\cap A}(B_1\cap
B_2,X) 
\la \hMap_{B_2\cap A}(B_2,X).
\]
This completes the proof.
\end{proof}

The following proposition demonstrates the use of the above results:

\begin{prop}
\label{pr:interval-mapping}
Suppose that $X\ra Y$ is a map of quasi-categories and that for all
$a,b\in X$ the induced map $\hMap_{\bdd{1}}(\Delta^1,X) \ra
\hMap_{\bdd{1}}(\Delta^1,Y)$ is a Kan equivalence.  Then for any pair of
$1$-simplices $f,g\colon \Delta^1\ra X$ (regarded as a single map
$\bdd{1}\times \del{1}\ra X$), the map
\[ \hMap_{\bdd{1}\times \Delta^1} (\Delta^1\times \Delta^1,X) \ra
\hMap_{\bdd{1}\times \Delta^1}(\Delta^1\times
\Delta^1,Y) \]
is also a Kan equivalence.
\end{prop}

\begin{proof}
Let $S$ and $T$ be the nondegenerate $2$-simplices $[00,01,11]$ and
$[00,10,11]$ in $\Delta^1\times \Delta^1$.  Write $S_0=[00,01]\cup
[11]$ and $T_0=[00]\cup [10,11]$.  Then by Proposition~\ref{pr:relative-pullback}
 the map we are
interested in is the induced map on pullbacks of the following
diagram:
\begin{myequation}
\label{eq:pb} 
\xymatrix{
\hMap_{S_0}(S,X) \ar@{->>}[r]\ar[d] & \hMap_{S_0\cap T_0}(S\cap T,X)\ar[d] &
\hMap_{T_0}(T,X) \ar@{->>}[l]\ar[d] \\
\hMap_{S_0}(S,Y) \ar@{->>}[r] & \hMap_{S_0\cap T_0}(S\cap T,Y) &
\hMap_{T_0}(T,Y). \ar@{->>}[l]
}
\end{myequation}
Now, $S_0\cap T_0\inc S\cap T$ is the inclusion of the boundary of a
$1$-simplex; hence the assumptions of the proposition imply that the
middle vertical map is a Kan equivalence.  

To analyze the left  vertical map in (\ref{eq:pb}) we let $S_1=[01,11]$
and consider the square
\[ \xymatrix{
\hMap_{S_0}(S,X) \ar[r] \ar[d] & \hMap_{S_0\cap S_1}(S_1,X) \ar[d]\\
\hMap_{S_0}(S,Y)\ar[r] &
\hMap_{S_0\cap S_1}(S_1,Y).
}
\]
As the map $S_1\amalg_{S_1\cap S_0} S_0 \inc S$ is a Joyal acyclic
cofibration, it follows by Proposition~\ref{pr:rel-fib} that the two
horizontal maps are Kan equivalences.  Finally, since $S_0\cap S_1\inc
S_1$ is the inclusion of the boundary of a $1$-simplex, the right
vertical map is a Kan equivalence by assumption.  Therefore the left
vertical map is a Kan equivalence as well.

A similar proof shows that the right vertical map in (\ref{eq:pb}) is
a Kan equivalence, which means this is true for all the vertical maps.
Since the horizontal maps are Kan fibrations, it follows that the
induced map on pullbacks is a Kan equivalence as well.  This is what
we wanted.
\end{proof}

%%%%%%%%%%%%%%%%%%%%%%%%%%%%%%%%%%%%%%%%%%%%%%%%%%%%%%%%%%%%%%%%%%%%%%

\section{DK-equivalences for quasi-categories}
\label{se:DK-equiv}
In this section we introduce a new notion of equivalence for
simplicial sets, called ``$DK$-equivalence."  It is designed to be
analagous to the corresponding notion for simplicial categories
\cite{Bergner}.  We will eventually prove that the class of
$DK$-equivalences is the same as the class of Joyal equivalences.
In the present section we set out to accomplish this by
establishing some basic properties of
$DK$-equivalences.

\medskip

\begin{defn}
\label{de:DK}
A map $f\colon X\ra Y$ of simplicial sets is said to be a
\dfn{DK-equivalence} if two conditions are satisfied:
\begin{enumerate}[(1)]
\item The induced map $\ho(\sSet_J)(*,X)\ra \ho(\sSet_J)(*,Y)$ is a
bijection;
\item For every two $0$-simplices $a,b\in X$, the induced map $
\hMap_{\bpsSet}(\Delta^1,X) \ra \hMap_{\bpsSet}(\Delta^1,Y)$ is a Kan equivalence.
\end{enumerate}
\end{defn}

The following lemma gives us three ways of recognizing $DK$-equivalences:

\begin{lemma}
\label{le:DK-char}
Let $f\colon X\ra Y$ be a Joyal fibration where both $X$ and $Y$ are
quasi-categories, and assume that $f$ satisfies condition (2) of
Definition~\ref{de:DK}.  Then the following statements are equivalent:
\begin{enumerate}[(a)]
\item $f$  has the right-lifting-property with
respect to $\emptyset\ra \Delta^0$ and $\{0,1\}\inc E^1$;
\item $[*,X]_{E^1} \ra [*,Y]_{E^1}$ is a bijection;
\item $f$ has the right-lifting-property with respect to $\emptyset
\ra \Delta^0$ (equivalently, $f$ is surjective).
\item $f$ satisfies condition (1) in Definition \ref{de:DK};
\end{enumerate}
\end{lemma}

\begin{proof}
Since $E^1$ is a cylinder object for $*$ in $\sSet_J$, (b) and (d) are equivalent.  We prove (a) $\Rightarrow$ (b) $\Rightarrow$ (c) $\Rightarrow$ (a).
The first implication is trivial, using the evident map of coequalizer diagrams defining $[-,-]_{E^1}$.

Assume (b) is true and
pick a map $a\colon
\Delta^0\ra Y$.  Since $[*,X]_{E^1}\ra [*,Y]_{E^1}$ is surjective, there is a map
$b\colon \Delta^0\ra X$ and a map $h\colon E^1\ra Y$ such that
$h(0)=f(b)$ and $h(1)=a$.  As $X\ra Y$ is a Joyal fibration, it has
the right-lifting-property with respect to $\{0\}\inc E^1$.  By lifting
the evident square we find a preimage for $a$ in $X$.  

Finally, assume (c) holds and suppose given a square
\[ \xymatrix{
\{0,1\} \ar[r]^{a\amalg b}\ar@{ >->}[d] & X \ar[d]^f \\
E^1\ar[r]^\gamma & Y.
}
\]
As $X\ra Y$ has the right-lifting-property with respect to $\{0\}\ra
E^1$, we can use this to lift $\gamma$ to a map $\beta\colon E^1\ra X$
satisfying $\beta(0)=a$.  Setting $a'=\beta(1)$, we then have
$f(a)=f(a')=\gamma(1)$.  

Let $F$ be the fiber of $f$ over $\gamma(1)$, and let $\bdd{1}\ra F$
send $0\mapsto a'$ and $1\mapsto b$.  Then we obtain a pullback square
\[ \xymatrix{
\bpsSet(C_E^\bullet,F_{a',b})  \ar[r]\ar[d] &
\bpsSet(C_E^\bullet,X_{a',b})\ar@{->>}[d]^\sim \\
\bpsSet(C_E^\bullet,\Delta^0)  \ar[r]\ &
\bpsSet(C_E^\bullet,Y_{\gamma(1),\gamma(1)})
}
\]
where the indicated map is a Kan acyclic fibration by our assumptions
on $f$.  As the pullback will also be a Kan acyclic fibration, this
shows that $\bpsSet(C_E^\bullet,F_{a',b})$ is contractible.  By
Corollary~\ref{co:connect1} we then know that $\jC(F)(a',b)$ is
contractible.

The same argument shows that $\jC(F)(a',a')$, $\jC(F)(b,b)$, and
$\jC(F)(b,a')$ are contractible, and this immediately yields that the
objects $a'$ and $b$ are isomorphic in the category $\pi_0\jC(F)$.
Hence by Corollary~\ref{co:3iso} there is a map $E^1\ra F$ connecting
$a'$ and $b$.  Let $\delta$ denote the composite $E^1\ra F\ra X$.

Let $h$ be the composite 
\[ E^2\llra{\pi} E^1 \llra{\gamma} Y \]
where $\pi$ sends $0\mapsto 0$, $1\mapsto 1$, and $2\mapsto
1$.  We then have a commutative square
\[ \xymatrix{
E^1\Wedge E^1\ar@{ >->}[d] \ar[r]^{\beta\Wedge \delta} & X\ar[d]^f \\
E^2\ar[r]^h & Y, 
}
\] 
where the left vertical map is the evident inclusion.  This map
is a Joyal acyclic
cofibration (both the domain and codomain are contractible in
$\sSet_J$), and so the above square has a lift $E^2\ra X$.  Restricting
to the vertices $0$ and $2$ gives a map $E^1\ra X$ which lifts our
original map $\gamma$.
\end{proof}

\begin{prop} 
\label{pr:DK}
Let $X$, $Y$, $X'$, and $Y'$ be quasi-categories.
\begin{enumerate}[(a)]
\item If $f\colon X\ra Y$ is a Joyal fibration and a $DK$-equivalence,
and $g\colon Y'\ra Y$ is any map, then the pullback
$X\times_Y Y' \ra Y'$ is
a Joyal fibration and a
$DK$-equivalence.
\item Let 
\[ \xymatrix{ X\ar@{->>}[r]^f\ar@{->>}[dr]_h & Y \ar@{->>}[d]^g \\
& Y'
}
\]
be a diagram in which all maps are Joyal fibrations.  Then if two of
the three maps are $DK$-equivalences, so is the third.
\item Consider a diagram
\[ \xymatrix{
X \ar@{->>}[r]\ar@{->>}[d] & Y \ar@{->>}[d] & Z\ar@{->>}[l]\ar@{->>}[d] \\
X' \ar@{->>}[r] & Y'  & Z'\ar@{->>}[l]
}
\]
where all the maps are Joyal fibrations and the vertical maps are
$DK$-equivalences.  Let $P$ and $P'$ be the pullbacks of the two rows,
and assume that $P\ra P'$ is also a Joyal fibration.  Finally, assume
that $X\ra X'\times_{Y'} Y$ and $Z\ra Z'\times_{Y'} Y$ are Joyal
fibrations.  Then $P\ra P'$ is also a $DK$-equivalence.
\end{enumerate}
\end{prop}

\begin{proof}
For (a), the pullback is certainly a Joyal fibration and its domain
$X\times_Y Y'$ is a quasi-category.  Since the map
$f$ has the RLP with respect to $\emptyset\ra \Delta^0$, so does the
pullback.  So we need only check condition (2) of
Defintion~\ref{de:DK}.  
Let $a=(a_1,a_2)$ and $b=(b_1,b_2)$ be
two points in $X\times_{Y} Y'$, where $a_1,b_1\in X$ and $a_2,b_2\in Y'$.
Let $C^\bullet$ be a cosimplicial resolution for $\Delta^1$ in $(\bpsSet)_J$.
Then we have a pullback square
\[ \xymatrix{
\bpsSet(C^\bullet,X\times_Y Y') \ar[d]\ar@{->>}[r] & \bpsSet(C^\bullet,Y')
\ar[d] \\
\bpsSet(C^\bullet,X)\ar@{->>}[r] & \bpsSet(C^\bullet,Y)
}
\]
where the indicated maps are Kan fibrations.  By assumption the bottom
horizontal map is a Kan equivalence, hence so is the top horizontal
map.  This is what we wanted.  

For (b), condition (1) of Definition~\ref{de:DK} clearly satisfies the
two-out-of-three property; we must check the same for condition (2).
This is trivial in the case where $f$ and $g$ are $DK$-equivalences,
and also in the case where $g$ and $h$ are $DK$-equivalences.  Assume
then that $f$ and $h$ are $DK$-equivalences. Given two points $a,b\in
Y$, we map lift them to $a',b'\in X$ because $f$ has the RLP with
respect to $\emptyset\ra \Delta^0$.  At this point the proof becomes
trivial again.

For (c), let $Q=X'\times_{Y'} Y$ and $R=Z'\times_{Y'} Y$.
The maps $Q\ra X'$ and $R\ra Z'$ are Joyal fibrations and
$DK$-equivalences by part (a).  Therefore the maps $X\ra Q$ and $Z\ra
R$, which are Joyal fibrations by assumption, are also
$DK$-equivalences by (b).  
In particular,
these maps have the right-lifting-property with respect to
$\emptyset\ra \Delta^0$.

Let $a$ be a $0$-simplex in $P'$.  This gives rise to $0$-simplices
$a_1\in X'$, $a_2\in Y'$, and $a_3\in Z'$.  Since $Y\ra Y'$ is a Joyal
fibration and $DK$-equivalence, there is a lift of $a_2$ to a point
$b_2$ in $Y$.  Then $(a_1,a_2,b_2)$ describes a $0$-simplex in $Q$,
hence there is a point $b_1$ of $X$ lifting it.  Likewise,
$(a_3,a_2,b_2)$ describes a $0$-simplex of $R$, so there is a point
$b_3$ in $Z$ lifting it.  The triple $(b_1,b_2,b_3)$ gives a
$0$-simplex of $P$ that lifts $a$.   Hence $P\ra P'$ has the
right-lifting-property with respect to $\emptyset\ra \Delta^0$.

To see that $P\ra P'$ satisfies condition (2) of
Definition~\ref{de:DK}, let $a=(a_1,a_2,a_3)$ and $b=(b_1,b_2,b_3)$
denote two points of $P$.  Let $C^\bullet$ denote any cosimplicial
resolution of $\Delta^1$ in $(\bpsSet)_J$.  Then we have a diagram
\[ \xymatrix{
\bpsSet(C^\bullet,X) \ar@{->>}[r]\ar@{->>}[d]^\sim & \bpsSet(C^\bullet,Y) 
\ar@{->>}[d]^\sim & \bpsSet(C^\bullet,Z) \ar@{->>}[l]\ar@{->>}[d]^\sim \\
\bpsSet(C^\bullet,X') \ar@{->>}[r] & \bpsSet(C^\bullet,Y')  & 
\bpsSet(C^\bullet,Z').\ar@{->>}[l]
}
\]
in which the indicated maps are Kan fibrations and Kan equivalences.
By a standard property of simplicial sets, the induced map on
pullbacks is a Kan equivalence.  But the two pullbacks coincide with
$\bpsSet(C^\bullet,P)$ and $\bpsSet(C^\bullet,P')$, so we have shown
that $\hMap_{\bpsSet}(\Delta^1,P)\ra \hMap_{\bpsSet}(\Delta^1,P')$ is a
Kan equivalence for every two points $a,b\in P$.  This completes the proof.
\end{proof}

\subsection{Comparing $DK$-equivalences and Joyal equivalences}

Clearly every Joyal equivalence is a $DK$-equivalence.  In this
section we will prove a partial converse, namely 
Proposition~\ref{pr:DKthenJoyal} 
below.  The complete converse is proven in Proposition~\ref{pr:allagree}.

\begin{lemma}
\label{le:DK2}
Let $X$ and $Y$ be quasi-categories and $f\colon X\fib Y$ be a Joyal
fibration and a $DK$-equivalence.
Then for every $n\geq 0$ the following
maps are also Joyal fibrations and $DK$-equivalences:
\begin{enumerate}[(a)]
\item $X^{\del{n}}\ra Y^{\del{n}}$;
\item $X^{\Lambda^n_k}\ra Y^{\Lambda^n_k}$ for any $0<k<n$;
\item $X^{\bdd{n}}\ra Y^{\bdd{n}}$;
\item $X^{\del{n}}\ra [Y^{\del{n}}\times_{Y^{\bdd{n}}} X^{\bdd{n}}]$.
\end{enumerate}
\end{lemma}

\begin{proof}
Note first that all the maps are Joyal fibrations between quasi-categories, by
Proposition~\ref{pr:Joyal-match}.
We next prove that $X^{\del{1}}\ra Y^{\del{1}}$ is a 
$DK$-equivalence.  
Condition (2) in the definition of $DK$-equivalence is verified by
Proposition~\ref{pr:interval-mapping}.  Using Lemma~\ref{le:DK-char}, it
will be enough to verify that $X^{\del{1}}\ra Y^{\del{1}}$ is
surjective.  That is, we must show that $X\ra Y$ has the
right-lifting-property with respect to $\emptyset\ra \Delta^1$.

Given a map $g\colon \Delta^1\ra Y$, we may lift $g(0)$ and $g(1)$ to
points $a$ and $b$ in $X$ since $X\ra Y$ is surjective by
Lemma~\ref{le:DK-char}.  As the map $\bpsSet(C_E^\bullet,X) \ra
\bpsSet(C_E^\bullet,Y)$ is a Kan acyclic fibration, and $g$ represents a
$0$-simplex in the target, we can lift $g$ to a $0$-simplex in the
domain.  This is what was wanted (recall that $C_E^0=\Delta^1$).

Now let $I[n]=\Spi[\Delta^n]$, so that $I[n]$ consists of $n$ copies of
$\Delta^1$ wedged together.  We next prove by induction that
$X^{I[n]}\ra Y^{I[n]}$ is a $DK$-equivalence for all $n\geq 1$.  The case
$n=1$ was handled above, so assume it is true for some $n\geq 1$.  Note
that $I[n+1]$ is a pushout of $I[n] \la \Delta^0 \ra \Delta^1$, and
therefore $X^{I[n+1]}$ is the pullback of $X^{I[n]}\ra X
\la X^{\del{1}}$.  Consider the diagram
\[ \xymatrix{
 X^{I[n]} \ar[r]\ar[d] & X \ar[d] & X^{\del{1}} \ar[d] \ar[l]\\
 Y^{I[n]} \ar[r] & Y  & Y^{\del{1}}.\ar[l]
}
\]
This diagram satisfies all the hypotheses of
Proposition~\ref{pr:DK}(c), therefore the induced map on pullbacks is
a $DK$-equivalence.  But this induced map is precisely $X^{I[n+1]}\ra
Y^{I[n+1]}$.

Now we turn to the proofs of (a) and (b).
The inclusion $I[n]\inc \del{n}$ is a Joyal acyclic cofibration, so
$X^{\del{n}}\ra X^{I[n]}$ is a Joyal acyclic fibration; in particular,
it is a $DK$-equivalence.  Applying Proposition~\ref{pr:DK}(b) (in
various ways) to the diagram
\[ \xymatrix{X^{\del{n}} \ar@{->>}[r]^{DK}\ar@{->>}[d] & X^{I[n]} \ar@{->>}[d]^{DK} \\
Y^{\del{n}} \ar@{->>}[r]_{DK} & Y^{I[n]}
}
\]
now gives that $X^{\del{n}}\ra Y^{\del{n}}$ is a $DK$-equivalence.  

For $0<k<n$, the inclusion $\Lambda^n_k\inc \del{n}$ is a Joyal
acyclic cofibration, and therefore $X^{\del{n}}\ra X^{\Lambda^n_k}$ is
a Joyal acyclic fibration.  A similar argument to the last paragraph
shows that $X^{\Lambda^n_k}\ra Y^{\Lambda^n_k}$ is a $DK$-equivalence.

We will prove part (c) by induction.  The cases $n=0$ and $n=1$
follow by hypothesis.  For $n\geq 2$ note that $\bdd{n}$ is the pushout of
$\del{n-1} \la \bdd{n-1}\ra \Lambda^n_{n-1}$.  This leads us to the
diagram
\[ \xymatrix{
 X^{\Lambda^n_{n-1}} \ar@{->>}[r]\ar@{->>}[d]_{DK} & X^{\bdd{n-1}} \ar@{->>}[d]
& X^{\del{n-1}} \ar@{->>}[d]^{DK} \ar@{->>}[l]\\
 Y^{\Lambda^n_{n-1}} \ar@{->>}[r] & Y^{\bdd{n-1}}  & Y^{\del{n-1}},\ar@{->>}[l]
}
\]
where the induced map on pullbacks is $X^{\bdd{n}}\ra Y^{\bdd{n}}$.
By parts (a) and (b) the indicated maps are $DK$-equivalences, and the
middle map is a $DK$-equivalence by induction.  One readily checks
that the diagram satisfies the conditions of
Proposition~\ref{pr:DK}(c), hence $X^{\bdd{n}}\ra Y^{\bdd{n}}$ is also
a $DK$-equivalence.  

Finally, to prove (d) let $n\geq 0$ and consider the diagram
\[\xymatrix{
X^{\del{n}}\ar@{->>}[dr]\ar@{->>}[drr]\ar@{->>}[ddr]\\
&P\ar@{->>}[r]\ar@{->>}[d] & X^{\bdd{n}} \ar@{->>}[d] \\
&Y^{\del{n}}\ar@{->>}[r] & Y^{\bdd{n}},
}
\]
where $P$ is the pullback.
We have proven that the right vertical map is a $DK$-equivalence,
hence so is the pullback by Proposition~\ref{pr:DK}(a).  We have also
proven that $X^{\del{n}}\ra Y^{\del{n}}$ is a $DK$-equivalence, so the
same is true for $X^{\del{n}}\ra P$ by Proposition~\ref{pr:DK}(b).
\end{proof}

\begin{prop}
\label{pr:DKthenJoyal}
If $X\ra Y$ is a $DK$-equivalence between quasi-categories and a Joyal
fibration then $X\ra Y$ is a Kan acyclic fibration (and so, in
particular, a Joyal equivalence).
\end{prop}

\begin{proof}
By Lemma~\ref{le:DK2}(d) we know that for any $n\geq 0$ the map
\[ {X}^{\del{n}}\ra {X}^{\bdd{n}}\times_{Y^{\bdd{n}}}
Y^{\del{n}}
\]
is a Joyal fibration and $DK$-equivalence.  In particular, it has the
right-lifting-property with respect to $\emptyset\ra \Delta^0$.  This
is equivalent to saying that any square
\[ \xymatrix{ \bdd{n}\ar[r]\ar[d] & {X} \ar[d]\\
\del{n}\ar[r] & Y
}
\]
has a lifting.
\end{proof}

%%%%%%%%%%%%%%%%%%%%%%%%%%%%%%%%%%%%%%%%%%%%%%%%%%%%%%%%%%%%%%%%%%%%%%

\section{Quillen equivalence of quasi-categories and simplicial
categories}
\label{se:Qequiv}
In this final section of the paper we use our previous results to
establish the equivalence between the homotopy theories of
quasi-categories and simplicial categories.  This result was
originally proven by Lurie \cite{L}.

\begin{prop}  
\label{pr:allagree}
For a map $X\ra Y$ of simplicial sets, the following are
equivalent:
\begin{enumerate}[(i)]
\item $f$ is a Joyal equivalence;
\item The map $\jC(f)\colon \jC(X)\ra \jC(Y)$ is a weak equivalence of
simplicial categories;
\item $f$ is a $DK$-equivalence.
\end{enumerate}
\end{prop}

\begin{proof}
The implication (i)$\Rightarrow$(ii) is  \cite[Proposition 6.6]{DS1}.

The equivalence of (ii) and (iii) can be argued as follows.  First,
note that both conditions are invariant under Joyal equivalence; that
is, if
\[ \xymatrix{ X\ar[r] \ar[d]_f & X' \ar[d]^{f'} \\
Y \ar[r] & Y'
}
\]
is a commutative square in which the horizontal maps are Joyal
equivalences, then $f$ satisfies (ii) (resp. (iii)) if and only if
$f'$ does.  For condition (iii) this is built into the definition,
whereas for condition (ii) it follows from the fact (i)$\Rightarrow$(ii).
It is therefore enough to prove that (ii) and (iii) are
equivalent under the assumption that $X$ and $Y$ are
quasi-categories.

But recall from Proposition~\ref{pr:isoclasses} that if $X$ is a quasi-category
then $[*,X]_{E^1}$ is in bijective correspondence with the isomorphism
classes in $\pi_0\jC(X)$.  Also, we know by
Corollary~\ref{co:connect1} that for any $a,b\in X$ the simplicial
set $\hMap_{\bpsSet}(\Delta^1,X_{a,b})$ is connected to $\jC(X)(a,b)$
by a natural zig-zag of Kan equivalences.  The equivalence of (ii) and
(iii) now follows at once from the definitions.
  
Finally, we prove (ii)$\Rightarrow$(i).  Let $Y\trcof \hat{Y}$ be a
fibrant-replacement in $\sSet_J$, and then factor the composite $X\ra
Y \ra \hat{Y}$ as a Joyal acyclic cofibration followed by a Joyal fibration.  This
produces a square
\[
\xymatrix{ X\ar[r]^f \ar[d]_\sim & Y \ar[d]^{\sim} \\
\hat{X} \ar@{->>}[r]^{\hat{f}} & \hat{Y}
}
\]
in which the vertical maps are Joyal equivalences (and therefore
become weak equivalences after applying $\jC$).  It follows that
$\hat{f}$ also becomes a weak equivalence after applying $\jC$, and
therefore $\hat{f}$ is a $DK$-equivalence by (ii)$\Rightarrow$(iii).
Then by Proposition~\ref{pr:DKthenJoyal} $\hat{f}$ is a
Joyal equivalence, and hence the same is true for the original map $f$
(by two-out-of-three applied to the above square).
\end{proof}

\begin{cor}\label{cor:main theorem}
The adjoint functors $\jC\colon \sSet_J\adjoint \sCat\colon N$ are a
Quillen equivalence.
\end{cor}

\begin{proof}
This is now easy, and we follow the same argument as in \cite{L}.  We
must prove two things, the first saying that for any fibrant
simplicial category $\cD$, the map
\[ \jC(N\cD) \ra \cD 
\]
is a weak equivalence in $\sCat$.  This has already been done in
Proposition~\ref{pr:counit}.

The second thing to be proven is that for any simplicial set $K$ and
any fibrant-replacement $\jC(K)\ra \cD$ in $\sCat$, the induced map
\[ K \ra N\jC(K) \ra N\cD
\]
is a Joyal equivalence.  By Proposition~\ref{pr:allagree}, it is
enough to show instead that
\[ \jC(K) \ra \jC N\jC(K)\ra \jC(N\cD) \]
is a weak equivalence in $\sCat$.  But now we consider the larger
diagram
\[ \xymatrix{
\jC(K) \ar[r]\ar@{=}[dr] & \jC N\jC(K) \ar[r]\ar[d] & \jC N\cD \ar[d]^\sim \\
& \jC(K) \ar[r]_{\sim} & \cD,
}
\]
where the right vertical map is a weak equivalence by
Proposition~\ref{pr:counit}.
It follows at once that the desired map is also a weak equivalence.
\end{proof}

%%%%%%%%%%%%%%%%%%%%%%%%%%%%%%%%%%%%%%%%%%%%%%%%%%%%%%%%%%%%%%%%%%%%%%

\section{Leftover proofs}
\label{se:leftover}
In this section we give two combinatorial proofs which were postponed
in Section~\ref{se:Dwyer-Kan}.

\medskip

The following situation is very useful.
Let $X$ be a simplicial set with the property that no two distinct
simplices have the same ordered sequence of vertices.  Examples of
such include $\Delta^n$, $E^1$, as well as subcomplexes and 
products of such things.  Suppose
that $X'$ is obtained from $X$ by pushing out $\Lambda^n_k\ra
\Delta^n$ along a map $\Lambda^n_k\ra X$; we say that $X'$ is obtained
from $X$ by filling a horn.  Then $X'$ will inherit the above property
of $X$ if
and only if $f$ is a non-bounding horn in the sense of the
following definition.

\begin{defn}
A map $\Lambda^n_i\to X$ is called a \dfn{non-bounding horn} in $X$ if
it does not extend to a map $\bd{n}\to X$.
\end{defn}

 See Section~\ref{se:connections} to
recall the notation used below.

\begin{lemma}
\label{le:necklace=simplex}
For any necklace $T$, the maps $\Spi[T]\inc T \inc \Delta[T]$ are both
Joyal equivalences.
\end{lemma}

\begin{proof}
We will argue that the map $\Delta^n\Wedge\Delta^1\to\Delta^{n+1}$ is
a composite of cobase changes along inner horn inclusions.  By
induction the same is therefore true for $\Spi[\Delta^{r}]\inc
\Delta^{r}$, for any $r>0$.  The map $\Spi[T]\inc \Delta[T]$ is
precisely one of these maps, therefore it is a Joyal equivalence.

The fact that $\Spi[T]\to T$ is a Joyal equivalence can then be proven,
bead by bead, using cobase changes of the maps $\Spi[\Delta^r]\to
\Delta^r$.  The desired result follows by two-out-of-three.

So everything follows once we have shown that $\Delta^n\Wedge \Delta^1
\inc \Delta^{n+1}$ is inner anodyne.  Let $X=\Delta^{n+1}$, and define
a filtration on $X$ by
\[ X_0=[01\ldots n]\cup [n,n+1], \quad
X_1=X_0\cup \bigcup_{i<n} [i,n,n+1], \quad X_2=X_1\cup
\bigcup_{i<j<n} [i,j,n,n+1],
\]
and so on.  That is, $X_0=\Delta^n\Wedge \Delta^1$ and $X_{i+1}$ is
the union of $X_i$ and all $(i+2)$-simplices of $X$ which
contain $n$
and $n+1$.  Note that $X_{n-1}=X$ and $X_{n-2}=\Lambda^{n+1}_n$.  

One readily checks that each inclusion $X_i\inc X_{i+1}$ is a cobase
change of $\binom{n}{i+1}$ inner horn inclusions.  For instance, the
inclusion $X_0\inc X_1$ is obtained by gluing the $2$-simplices
$[i,n,n+1]$  to $X_0$ along their inner horns $\Lambda^{\{i,n,n+1\}}_n$,
one $2$-simplex for each $i\in \{0,1,\ldots,n-1\}$.  The inclusion $X_1\inc X_2$
is obtained by gluing the $3$-simplices $[i,j,n,n+1]$  to $X_1$ along their inner
horns $\Lambda^{\{i,j,n,n+1\}}_n$, and so forth.
\end{proof}

Our next goal is to complete the proof of
Proposition~\ref{pr:three-resolutions} by showing that for any $n\geq
1$ the canonical maps $C^n_R \ra \Delta^1$, $C^n_L\ra
\Delta^1$, and $\Ccyl^n\ra \Delta^1$ are all Joyal equivalences.  The
proof will proceed by a combinatorial argument similar to the above.

Given integers $0\leq k< n$, define $\Delta^n_k$ to be the
quotient of $\Delta^n$ obtained by collapsing the initial $\Delta^k$
to a point and the terminal $\Delta^{n-k-1}$ to a (different) point.  Note that
$\Delta^n_k$ has exactly two vertices, and there is a unique
surjection $\Delta^n_k \ra \Delta^1$.  Note also that
$\Delta^n_{n-1}=C^n_R$ and $\Delta^n_{0}=C^n_L$.  

\begin{lemma}
\label{le:squash}
For integers $0\leq k< n$, 
the surjection $\Delta^n_k \ra \Delta^1$ is a Joyal equivalence.
\end{lemma}

\begin{proof}
We first do the case $k=0$.  Let $X=\Delta^n_0$, and note that every non-degenerate
simplex $\sigma$ of $X$ is the image of a unique non-degenerate simplex in
$\Delta^n$; hence we can denote $\sigma$ by the vertices of its preimage.  Define a
filtration on $X$ by
\[ X_0=[01], \quad X_1=\bigcup_{1<i\leq n} [01i], \quad X_2=\bigcup_{1<i<j\leq n}
[01ij],
\]
and so on.  Note that $X_{n-1}=X$.  It is easy to check that $X_i\inc
X_{i+1}$ is a cobase change along $\binom{n-1}{i+1}$ inner horn
inclusions ($0\leq i\leq n-2$), and therefore is a Joyal equivalence.  So $X_0\inc X$ is
also a  Joyal equivalence, and the desired result follows by
two-out-of-three.  

The proof in the case $k=n-1$ is completely symmetric to the $k=0$ case.

It remains to tackle the case $0<k<n-1$.  Consider the inclusion
$\Delta^{n-1}\inc \Delta^n$ given by $[12\ldots n]$, and the induced
inclusion $\Delta^{n-1}_{k-1} \inc \Delta^n_k$.  We may assume by
induction that $\Delta^{n-1}_{k-1}\ra \Delta^1$ is a Joyal
equivalence, so it suffices to prove the same for
$\Delta^{n-1}_{k-1}\inc \Delta^n_k$.  We will prove that this map is
inner anodyne.

Define a filtration on $X=\Delta^n_k$ by $X_0=\Delta^{n-1}_{k-1}$ and
\[ X_1=X_0\cup \!\!\bigcup_{k<j_1} [01j_1], \ \ \ X_2=X_1\cup\!\! \bigcup_{k<j_1<j_2}
[01j_1j_2], \ \ \ 
X_3=X_2\cup \!\!\!\!\!\!\bigcup_{k<j_1<j_2<j_3} [01j_1j_2j_3],
\]
and so on.
So $X_i$ is the union of $X_{i-1}$ and all $(i+1)$-simplices of $X$
containing $0$ and $1$.  It is again easy to see that each $X_i\inc
X_{i+1}$ is a cobase change along non-bounding inner horn inclusions (the $1$-horn
in each case), and so $X_0\inc X_{n-1}=X$ is inner anodyne.  
\end{proof}

\begin{prop}\label{prop:squash}
For every $n\geq 0$, the maps $C^n_R\ra\Delta^1$, $C^n_L\ra \Delta^1$,
and $\Ccyl^n\ra \Delta^1$ are Joyal equivalences. 
\end{prop}

\begin{proof}
The cases of $C^n_R$ and $C^n_L$ follow immediately from
Lemma~\ref{le:squash}.  For $\Ccyl^n$ we argue as follows.
Let $\{0,1,\ldots,n\}$ and $\{0',1',\ldots,n'\}$ denote the vertices
in $\Delta^n\cross\{0\}$ and $\Delta^n\cross\{1\}$, respectively.
Note that each simplex of
 $\Delta^n\cross\Delta^1$ is completely determined by its vertices,
and that
$\Delta^n\cross\Delta^1$ contains exactly
$(n+1)$ non-degenerate $(n+1)$-simplices.  For example if $n=2$ we
have:
\[
\Delta^2\cross\Delta^1=
\Delta^{\{0,1,2,2'\}}\cup \Delta^{\{0,1,1',2'\}}
\cup \Delta^{\{0,0',1',2'\}}.
\]
Let $D_i$ denote the $(n+1)$-simplex
$\Delta^{\{0,1,\ldots,i,i',\ldots,n'\}}$, for $0\leq i\leq n$.  

Recall that $\Ccyl^n$ is obtained from $\Delta^n\cross\Delta^1$ by
collapsing any simplex whose vertices are contained in
$\{0,\ldots,n\}$ or in $\{0',\ldots,n'\}$.
Let $E_i$ be the image of $D_i$ in $\Ccyl^n$, and note that $E_i$ is
isomorphic to   
$\Delta^{n+1}_{i}$.  

Define a filtration on $X=\Ccyl^n$ by setting 
\[ X_i=E_0\cup E_1 \cup \cdots \cup E_i.
\]
Note that $X_n=X$.  We will prove by induction that the composite
$X_i\inc X \ra \Delta^1$ is a Joyal equivalence, for every $i$.  The
base case $i=0$ is covered by Lemma~\ref{le:squash}.  

Note that $X_{i+1}=X_i\cup E_{i+1}$, and $E_{i+1}\cap X_i \iso
\Delta^{\{0,1,\ldots,i,(i+1)',\ldots,n'\}}_i$.    
This gives us a diagram in which both squares are pushouts:
\[ \xymatrix{
\Delta^{n}_i \ar@{=}[r]\ar@{ >->}[d] & E_{i+1}\cap X_i \ar[r]\ar[d] & X_i \ar[d]\\
\Delta^{n+1}_i \ar@{=}[r] & E_{i+1} \ar[r] & X_{i+1}.
}
\]
The left vertical map is a Joyal acyclic cofibration (using
Lemma~\ref{le:squash} and two-out-of-three), so the same is true for $X_i\inc X_{i+1}$.
This completes the proof.
\end{proof}

%%%%%%%%%%%%%%%%%%%%%%%%%%%%%%%%%%%%%%%%%%%%%%%%%%%%%%%%%%%%%%%%%%%
\appendix

%%%%%%%%%%%%%%%%%%%%%%%%%%%%%%%%%%%%%%%%%%%%%%%%%%%%%%%%%%%%%%%%%

\section{The Box-product lemmas}
\label{se:box}

In these appendices we develop all the properties of quasi-categories
used in this paper, from first principles.  This material is a summary
of
Joyal's work \cite{J2}, and  culminates in the
proof of the existence of the Joyal model category structure
(Theorem~\ref{th:model}). 
Note that
the only results in this paper which directly rely on the appendices
occur in Section~\ref{se:quasi}.

Appendix~\ref{se:box} is totally self-contained: we prove two
box product lemmas, one yielding Proposition~\ref{pr:box-inner} and
the other serving as a key step for Proposition~\ref{pr:quasi1}(a).
Appendix~\ref{se:special} is also self-contained: there  we prove
every unjustified result of Section~\ref{se:quasi} except for
Theorem~\ref{th:model}.  Appendix~\ref{se:model} proves this final
theorem.

\medskip

\begin{lemma}
The box product
$ (\Lambda^n_k\inc \Delta^n)\bbox (\bdd{r}\inc \del{r})$
is inner anodyne when $0<k<n$ and $0\leq r$.
\end{lemma}

\begin{proof}
Let $Y=\Delta^n\cross\Delta^r$ and let 
$Y_0=(\Lambda^n_k\cross\Delta^r)\amalg_{\Lambda^n_k\cross\bd{r}}(\Delta^n\cross\bd{r})$.
We will produce a filtration $Y_0\subseteq Y_1\subseteq\cdots\subseteq Y_{r+1}=Y$
and prove that each $Y_i\inc Y_{i+1}$ is inner anodyne.

Let us establish some notation.  An $m$-simplex $y$ in $Y$ is
determined by its vertices, and we can denote it in the form
\begin{myequation}
\label{eq:y}
y=\begin{bmatrix}
a_0&a_1&\ldots&a_m\\u_0&u_1&\ldots&u_m\end{bmatrix}
\end{myequation}
where $0\leq a_i\leq a_{i+1}\leq n$ and $0\leq u_i\leq u_{i+1}\leq r$,
for $0\leq i<m$.  Here we are writing $a\brack u$ instead of the usual
$(a,u)$, because in the above notation the faces and degeneracies are
obtained by just omitting or repeating columns.  The simplex $y$ is
degenerate if and only if two successive columns are identical.  Note
that two distinct non-degenerate simplices of $Y$ will not have any
horn in common.

One checks that the simplex $y$
is an element of $Y_0$ if and only if it satisfies one of the following
two conditions:
\begin{enumerate}
[\hspace{.15in}(i)]\item
$k\in\{a_0,a_1,\ldots,a_m\}\neq\{0,1,\ldots,n\},$ \hspace{.2in} -OR-
\item $\{b_0,\ldots,b_m\}\neq\{0,1,\ldots,r\}$.
\end{enumerate}

Let $Y_1$ be the union of $Y_0$ together with all simplices that
contain the vertex $k\brack 0$, and in general let $Y_i$ be the union of $Y_{i-1}$
together with all simplices containing $k\brack i-1$.  Note that
$Y_{r+1}=Y$: this follows from the fact that
every simplex of $Y$ is a face of a simplex that contains some
$k\brack i$.  

Our goal is to show that each inclusion $Y_i\inc Y_{i+1}$ is inner
anodyne, and we will do this by producing another filtration
\[ Y_i=Y_i[n-1]\subseteq Y_i[n] \subseteq \cdots \subseteq
Y_i[n+r]=Y_{i+1}. \]
Notice that every simplex of $Y$ of dimension $n-1$ or less, containing
$k\brack i$, lies in $Y_0$ (it satisfies condition (i)).  For $t>n-1$
we define $Y_i[t]$ to be the union of $Y_i[t-1]$ and all nondegenerate
simplices of $Y$ that have dimension $t$ and contain $k\brack i$.  We
claim that $Y_i[t]\inc Y_i[t+1]$ is a cobase change of a disjoint
union of inner horn inclusions; justifying this will conclude our
proof.

Let $y$ be a nondegenerate simplex of $Y$ of dimension $t+1$,
where $t\geq n-1$, and assume $y$ contains $k\brack i$ but
$y\notin Y_i[t]$.  
Then every
face of $y$ except possibly for the $k\brack i$-face is contained
in $Y_i[t]$.  We must show that $Y_i[t]$ cannot contain this final
face of $y$, and also that this final face is not simultaneously
filled by another inner horn in $Y_i[t]$;  the latter is clear by an above note, so we will
concentrate on the former.

Write $y$ in the form of (\ref{eq:y}), and consider the column
immediately preceding the $k\brack i$.  This column cannot have $k$ in
the top row, for then the bottom row would be at most $i-1$ and we would have
 $y\in Y_{i-1}$.  Likewise, the top
entry of this column cannot be $k-2$ or less because otherwise we
would have $y\in Y_0$.  So immediately preceding the $k\brack i$-column is a
$k-1\brack ?$-column.  
This cannot be a $k-1\brack j$-column for
$j<i$, since otherwise $y$ would be a face of the $(t+2)$-dimensional 
simplex obtained by inserting a $k\brack i-1$-column before the
$k\brack i$; this would imply that $y\in Y_{i-1}$, a contradiction.
So the $k\brack i$-column is preceded by $k-1\brack i$.

Now consider the $k\brack i$-face of $y$; call this face $dy$.
Note that the {\it set\/} of entries in the second row of $dy$ is the
same  
as the corresponding set for the second row of $y$.  A little thought
then shows that $dy$ cannot
satisfy condition (i) or (ii) above, since $y$ did  not.  So $dy\notin
Y_0$.   Then the only way $dy$ could be in $Y_{i-1}$ is if it were
added somewhere along the way, but every simplex that was added was
part of a simplex containing a $k\brack j$ for $j<i$; and clearly $dy$
is not part of any such simplex.  So $dy\notin Y_{i-1}$.  Finally, the
$t$-dimensional simplices in $Y_i[t]-Y_{i-1}$ all contain $k\brack i$,
and therefore $dy$ is not one of these either.  Hence $dy\notin
Y_i[t]$, and this shows that adjoining $y$ to $Y_i[t]$ amounts to
filling a non-bounding horn, which is inner because $k\notin\{0,n\}$.
 This completes the argument.
\end{proof}

The previous lemma allows us to wrap up a loose end from Section~\ref{se:quasi}.

\begin{proof}[Proof of Proposition~\ref{pr:box-inner}]
This follows immediately from the preceding lemma, by a standard
argument.
\end{proof}

We next turn to box products of maps with $\{0\}\inc E^1$.  
Recall that in Section~\ref{se:quasi} we introduced the notion of a
quasi-isomorphism in a quasi-category $X$.  It is convenient to extend
this notion to all simplicial sets $S$ by saying that a $1$-simplex
$e$ in $S$ is a quasi-isomorphism if for every quasi-category
$X$ and every map $S\ra X$, the image of $e$ in $X$ is a
quasi-isomorphism as previously defined.  Note that if $e\colon\Delta^1\ra S$
extends to a map $E^1\ra S$, then $e$ is necessarily a
quasi-isomorphism.  

\begin{defn}

A map $\ell\taking\Lambda^n_k\to X$ is called a \dfn{special right
horn} (resp. \dfn{special left horn}) if $k=n$ (resp. $k=0$) and
$\ell(\Delta^{\{n-1,n\}})$ (resp. $\ell(\Delta^{\{0,1\}})$) is a
quasi-isomorphism in $X$.  A \dfn{special outer horn} is defined to be
either a special left horn or a special right horn.  Finally $\ell$ is
a \dfn{special horn} if it is either a special outer horn or any inner
horn.

A map $f\taking X\to Y$ is called \dfn{special outer anodyne} if it is
the composition of cobase extensions along special outer horns.  The
map $f$ is called \dfn{special anodyne} if it is the composition of
cobase extensions along inner horns and special outer horns.

\end{defn}

The next box product lemma will be a key step in proving
Proposition~\ref{pr:quasi1}(a).  Note the restriction of $r\geq 1$,
which will be important later.

\begin{lemma}
\label{le:box-special}
For any $r\geq 1$ the box product $f=(\{0\}\inc 
E^1)\square(\bd{r}\to\Delta^r)$ is special anodyne.
\end{lemma}

\begin{proof}
Let $Y=E^1\cross\Delta^r$ and let
$Y_0=(\{0\}\cross\Delta^r)\amalg_{\{0\}\times \bd{r}}(E^1\cross\bd{r})$.
We will produce a filtration of simplicial sets
$Y_0\subseteq Y_1\subseteq \cdots \subseteq Y_{r+1}=Y$
such that $Y_i\inc Y_{i+1}$ is inner anodyne for $0<i<r$ and special
outer anodyne for $i=0$ and $i=r$.

Let us establish some notation.  It is convenient to denote the
$0$-simplices of $E^1$ by $a$ and $b$ rather than $0$ and $1$.  Note
that every simplex of $E^1$ is uniquely determined by its vertices,
and so the same is true of $Y$.
The vertices of $Y$ are of the form
$(a,i)$ or $(b,i)$ for $0\leq i\leq r$; to ease the typography we
will denote these $a_i$ and $b_i$, respectively.  Since a simplex $y$
of $Y$ is determined by its ordered set of vertices, we can denote it
\begin{myequation}
\label{eq:y2}
y=\Bigl
([y_0^1,\ldots,y_0^{k_0}],[y_1^1,\ldots,y_1^{k_1}],\ldots,[y_r^1,\ldots,y_r^{k_r}]
\Bigr )
\end{myequation}
where for each $0\leq i\leq r$ we have $0\leq k_i$ and each $y_i^j$ is
either the vertex $a_i$ or $b_i$.  Note that the superscripts only
serve as counters.  Similarly, the brackets are not part of the data
here; they are written for the ease of the reader.  

Said differently, the simplex $y$ corresponds to a sequence of $a_0$'s
and $b_0$'s, followed by a sequence of $a_1$'s and $b_1$'s, and so on
up through the final sequence of $a_r$'s and $b_r$'s.  We refer to the
portion of the sequence consisting of the $a_i$'s and $b_i$'s as the
``$i$-group''; note that this can be empty.  The $i$-group corresponds
to the sequence inside of the $i$th set of brackets in (\ref{eq:y2}).
Note that the simplex $y$ is
degenerate if there exists $0\leq i\leq r$ and $1\leq j\leq k_i-1$
such that $y_i^j=y_i^{j+1}$ --- that is, if there is a repetition inside
one of the groups. 

Let $Y_m$ denote the simplicial subset of $Y$ consisting of all
simplices $y\in Y$ such that one of the following conditions is
satisfied:
\begin{enumerate}[\hspace{.15in}(i)]\item there exists
$0\leq i\leq r$ such that $k_i=0$ (i.e., one of the groups is empty) \hspace{.2in} -OR- 
\item for all
$i\geq m$ and all $1\leq j\leq k_i$, one has
$y_i^j=a_i$ (the $m$-group and higher consist only of $a$'s).
\end{enumerate} 
Note that $Y_0$ agrees with our previous definition, and $Y_{r+1}=Y$
because condition (ii) is vacuously satisfied when $m=r+1$.  

Consider the inclusion $Y_m\inc Y_{m+1}$ for some $1\leq m\leq r-1$
(the cases $m=0$ and $m=r$ will be handled separately).
The simplices in $Y_{m+1}-Y_m$ are of the form
\[x=([x_0^1,\ldots,x_0^{k_0}],\ldots,[x_m^1,\ldots,x_m^{k_m}],
[a_{m+1}^1,\ldots,a_{m+1}^{k_{m+1}}],\ldots,[a_r^1,\ldots,a_r^{k_r}])
\]
where $a_i^j=a_i$ and where there is at least one $b$ in the
$m$-group.
Note that every such simplex is the face of a simplex of the form
\[
x'=([x_0^1,\ldots,x_0^{k_0}],\ldots,[x_m^1,\ldots,x_m^{k_m}a_m],
[a_{m+1}^1,\ldots,a_{m+1}^{k_{m+1}}], \ldots,[a_r^1,\ldots,a_r^{k_r}]).
\]
Define an infinite sequence 
\[ Y_m=Y_m[0]\subseteq Y_m[1]\subseteq Y_m[2] \subseteq \cdots
\]
whose union is $Y_{m+1}$ by letting $Y_m[t]$ be the union of
$Y_m[t-1]$ and all $t$-simplices of the form $x'$ above.

We claim that each $Y_m[t]\inc Y_m[t+1]$ is a cobase change of a
coproduct of inner horn inclusions.  To do this we'll show that every
nondegenerate simplex $x'\in Y_m[t+1]-Y_m[t]$ comes from a unique 
non-bounding horn in $Y_m[t]$.  

By the ``$a_m$-face'' of $x'$ we mean the face corresponding to the
final $a_m$ in the $m$th group.  It is clear that every other face of
$x'$ lies in $Y_m[t]$, so we need only show that this $a_m$-face does
not lie in $Y_m[t]$.  But the $m$-group of $x'$ must contain at least
one $b$, and this shows that the $a_m$-face of $x'$ is not in
$Y_{m}$.  The only way it could be in $Y_m[t]$ is as a face of a
simplex whose $m$-group ends in $a$, and clearly this is not possible for dimensional reasons.

This completes our analysis of $Y_m\inc Y_{m+1}$ for $0<m<r$.  When
$m=0$ or $m=r$ the same idea works, but the horns involved are special
outer horns (they are special because every $1$-simplex within a
single group---in particular, the $0$th group or the $r$th group---is
a quasi-isomorphism).  In fact one literally copies the previous
paragraphs, replacing all instances of the word ``inner" with the
phrase ``special outer."  There is one subtlety that occurs, which is
why one needs $r>0$.  In passing from $Y_0$ to $Y_1$, the first stage
of the argument involves attaching
the simplex $[b_0a_0a_1\ldots a_r]$ along its $0$-horn.  But in the
case $r=0$ this is a $0$-horn of a $1$-simplex, which is not allowed.
This problem does not appear when $r>1$, and so
this completes the proof.
\end{proof}

%%%%%%%%%%%%%%%%%%%%%%%%%%%%%%%%%%%%%%%%%%%%%%%%%%%%%%%%%%%%%%%%%%%%%%%

\section{Special outer horns, and applications}
\label{se:special}
Quasi-categories do not satisfy the lifting condition for general outer
horns.  But it turns out they {\it do\/} satisfy the lifting condition
for special outer horns---outer horns where a particular map is a
quasi-isomorphism.  The purpose of this section is to prove various
lifting results related to this phenomenon.

\medskip

\subsection{The quasi-isomorphism lemmas}

Let $X$ be a quasi-category and let $f$ be a $1$-simplex in $X$.
Recall that a $1$-simplex $h$ is a \dfn{right inverse} (or right
quasi-inverse) for $f$ if there exists a $\sigma\ra\Delta^2\ra X$ such
that $d_0(\sigma)=f$, $d_1(\sigma)$ is degenerate, and
$d_2(\sigma)=h$.  We call the $2$-simplex $\sigma$ a \dfn{right
inverse provider} for $f$.

Similarly,  $h$ is a \dfn{left inverse} for $f$ if there
exists a $\tau\colon \Delta^2\ra X$ with $d_2(\tau)=f$,
$d_1(\tau)$ degenerate, and $d_0(\tau)=h$; and $\tau$ is called a
\dfn{left inverse provider} for $f$.

In Proposition~\ref{pr:tau_1} we stated that if $X$ is a
quasi-category and $f$, $g$, and $h$ are 1-simplices in $X$, then $gf=h$ in
$\pi_0\jC(X)$ if and only if there is a $2$-simplex $\Delta^2\ra X$
with $d_0=g$, $d_1=h$, and $d_2=f$.  Here is the proof:

\begin{proof}[Proof of Proposition~\ref{pr:tau_1}]
Let $X$ be a quasi-category.
Define a relation on the $1$-simplices of $X$ by $f\sim g$ if there
exists a map $\sigma\colon \Delta^2\ra X$  with $d_1(\sigma)=g$,
$d_2(\sigma)=f$, and $d_0(\sigma)$ is degenerate.  It is proven in
\cite[Lemma 4.11]{BV} (as well as \cite{J2}) that this gives an
equivalence relation, and that there is a category $\ho(X)$ where the
maps are equivalence classes of $1$-simplices.  It is then easy to
produce maps of categories $\ho(X)\ra \pi_0\jC(X)$ and $\pi_0\jC(X)\ra
\ho(X)$ showing that the two categories are isomorphic.

Now suppose that $f$, $g$, and $h$ are $1$-simplices of $X$ and that
$gf=h$ in $\pi_0\jC(X)$.  Let $\sigma\colon \Delta^2\ra X$ be any
$2$-simplex with $d_0(\sigma)=g$ and $d_2(\sigma)=f$ (such a simplex
exists by the quasi-category condition).  Let $u=d_1(\sigma)$. 
Then $u$ and $h$ will represent the same map in
$\pi_0\jC(X)$, so we must have $u\sim h$.  
Thus, there is a $\tau\colon \Delta^{\{0,2,3\}}\ra X$ with
$d_2(\tau)=h$, $d_3(\tau)=u$, and $d_0(\tau)$ degenerate.  
We obtain in this way a map $F\colon \Lambda^3_2\ra X$ which equals
$\sigma$ on $[012]$, $\tau$ on $[023]$, and is a degeneracy of $g$ on $[123]$.
Extending $F$ to $\Delta^3$ and restricting to $[013]$ gives the
desired $2$-simplex with boundary $(g,h,f)$. 
\end{proof}

\begin{prop}
\label{pr:3iso-again}  
Let $X$ be a quasi-category, and let $f\colon
\Delta^1\ra X$.  Then the following
conditions are equivalent:
\begin{enumerate}[(i)]
\item $f$ is a quasi-isomorphism;
\item $f$ has a left inverse  and a (possibly different)  right
inverse.
\item The image of $f$ in $\pi_0\jC(X)$ is an isomorphism.
\end{enumerate}
\end{prop}

\begin{proof}
Clearly (i)$\Rightarrow$(ii)$\Rightarrow$(iii), and
(iii)$\Rightarrow$(i) follows immediately from Proposition~\ref{pr:tau_1}. 
\end{proof}

%We will also have need of the following:
%
%
%\begin{lemma}\label{lemma:right-left inverse}
%Let $X$ be a quasi-category and 
%suppose that $f\in X_1$ is a quasi-isomorphism.  Then any right inverse
%$h\in X_1$ of $f$ is also a left inverse (and vice-versa).
%\end{lemma}
%
%\begin{proof}
%We prove that right inverses are left inverses; the dual
%statement is similar.
%
%Let $\sigma$ denote a $2$-simplex with $d_0(\sigma)=f$,
%$d_2(\sigma)=h$, and $d_1(\sigma)$ degenerate.  
%Since $f$ is a
%quasi-isomorphism it does have some left inverse: that is, 
%there exist a $2$-simplex $\tau$ such that $d_2(\tau)=f$ and
%$d_1(\tau)$ is degenerate.  Let $d_0(\tau)=h'$.
%
%Define $G\taking\Lambda^3_1\to X$ by
%\[
%G([0,1,2])=\sigma, \quad G([0,1,3])=s_1(h),\quad\textnormal{ and }\quad
%G([1,2,3])=\tau.
%\]
%Since $X$ is a quasi-category, we can fill this horn to get a simplex
%$H\taking\Delta^{\{0,1,2,3\}}\to X$.  Note that the new 2-simplex
%$\theta=H|_{[023]}$ has faces $\theta|_{[03]}=h$, $\theta|_{[23]}=h'$,
%and $\theta|_{[02]}$ is degenerate.
%
%Next define a map $G'\taking\Lambda^{\{-1,0,2,3\}}_2\to X$ by
%\[ G'([0,2,3])=\theta,\quad
%G'([-1,2,3])=\tau,\quad  G'([-1,0])=f,
%\]  
%and take $G'([-1,0,2])$ to be the evident degeneracy of $f$.
%Since $X$ is a quasi-category there is an
%extension of $G'$ to
%$H'\taking\Delta^{\{-1,0,2,3\}}\to X$.  The face $H'|_{[-103]}$ is a 2-simplex with
%$H'([-1,0])=f$, $H'([0,3])=h$, and $H'([-13])$ is degenerate.  This 
%shows that $h$ is a left inverse of $f$.
%\end{proof}

\subsection{Special outer horn lifting}
Suppose given a map $f\colon \Lambda^n_0\ra X$ such that $f([01])$ is
a quasi-isomorphism.  Since $f$ has an inverse, one could imagine
producing a corresponding horn in which the direction of $f$ has been
``flipped''.  This would be an inner horn, which could be extended to an
$n$-simplex $\Delta^n\ra X$.  One could imagine flipping the $[01]$
simplex again to create an extension of the original map $f$.  

Although vague, the above paragraph gives an idea of why
quasi-categories should have liftings for special outer horns.
Attempting to prove this result in the above manner results in a
combinatorial nightmare.  A different approach, using some clever
techniques of Joyal \cite{J2}, allows one to simplify it.

Given two maps $f\colon A\inc B$ and $g\colon C\inc D$, let us use the notation
$f\jbox g$ for the map
\[ (A\join D)\amalg_{A\join C} (B\join C) \inc B\join D.
\]

\begin{lemma}
\label{le:joinbox}
Let $A\inc B$ be a monomorphism of simplicial sets.  Then for any
$n\geq 0$ and any $0<k\leq n$, the map
\[ (\Lambda^n_k\inc \Delta^n) \jbox (A\inc B) \]
is inner anodyne.  (Note that the case $k=n$ is allowed.) 
\end{lemma}

\begin{proof}
By a routine argument one reduces to the case where $A\inc B$ is
$\bdd{r}\inc \del{r}$.  So we are looking at the map
\begin{myequation}
\label{eq:mmap}
(\Lambda^n_k\join \del{r})\amalg_{\Lambda^n_k\join \bdd{r}}
(\del{n}\join \bdd{r}) 
\inc \del{n}\join \del{r}.
\end{myequation}
Note that the codomain may be identified with
$\Delta^{\{0,1,\ldots,n,n+1,\ldots,n+r+1\}}$.  Under this
identification $\del{n}\join \bdd{r}$ is the union of the faces
$[01\ldots n,n+1,\ldots\hat{i}\ldots,n+r+1]$ where $n+1\leq i\leq
n+r+1$.  Likewise, the simplicial set $\Lambda^n_k \join \del{r}$ is the union of the faces
$[01\ldots\hat{i}\ldots n,n+1\ldots n+r+1]$ where $0\leq i\leq n$
and $i\neq k$.  It follows that the domain of the map in
(\ref{eq:mmap})
is just $\Lambda^{\{0,1,\ldots,n+r+1\}}_k$, and so our map is an inner
horn (even if $k=n$ or $r=0$).  
\end{proof}

Let $K$ be a simplicial set.  As pointed out in \cite[Section 5.2]{J2},
the join functor
$(\blank)\join K\colon \sSet \ra (K\ovcat \sSet)$ has a right
adjoint.  This right adjoint sends a simplicial set $X$ with a map
$K\ra X$ to a simplicial set denoted $X_{/K}$; the $n$-simplices of
$X_{/K}$ are the maps $\Delta^n\join K\ra X$ which extend the given
map $K\ra X$.  Note that $X_{/\emptyset}=X$.  

The following lemma will be useful:

\begin{lemma}
\label{le:slices of qcats are qcats}
Let $A\ra X$ be a map of simplicial sets and let $X\ra Y$ be an inner fibration.
Then $X_{/A}\ra Y_{/A}$ is also an inner fibration.  In particular, if
$X$ is a quasi-category then so is $X_{/A}$.
\end{lemma}

\begin{proof}
This is immediate from Lemma~\ref{le:joinbox}, using adjointness.
\end{proof}

\begin{prop}
\label{pr:conservative}
Let $X$ be a quasi-category and let $K\ra X$ be a map of simplicial sets.  Given
$f\colon \Delta^1\ra X_{/K}$, let $f'$ denote the composite $\Delta^1\ra
X_{/K}\ra X_{/\emptyset}=X$.  Then if $f'$ is a quasi-isomorphism, so is $f$.

\end{prop}

\begin{proof}
We will prove that $f$ has a right inverse $g$.  Then we will prove
that $g$ itself has a right inverse.  Therefore $g$ has a left and
right inverse, hence is a quasi-isomorphism by
Proposition~\ref{pr:3iso-again}.  Then $f$ is a quasi-isomorphism by
two-out-of-three.

Choose a right inverse provider $\sigma\colon \Delta^{\{-1,0,1\}} \ra
X$ for $f'\taking\Delta^{\{0,1\}}\to X$: so $\sigma|_{[01]}=f'$ and $\sigma|_{[-11]}$ is degenerate.  
Consider the lifting diagram
\begin{myequation}
\label{dia:conservative}\xymatrixcolsep{2.6pc}\xymatrix{
(\Delta^{\{-1,0,1\}} \star \emptyset) \cup ((\Delta^{\{-1,1\}} \cup
\Delta^{\{0,1\}})\join K) \ar[r]^-F \ar@{ >->}[d] & X \\
\Delta^{\{-1,0,1\}}\join K \ar@{.>}[ur]_G
}
\end{myequation}
where the map $F$ is given by $\sigma$ on the left piece, the adjoint
of $f$ on the $\Delta^{\{0,1\}}\join K$ piece, and on the remaining
piece is equal to the composite
\[ \Delta^{\{-1,1\}}\join K \ra \Delta^{\{1\}} \join K \inc
\Delta^{\{0,1\}}\join K \llra{f} X.
\]
The vertical map in (\ref{dia:conservative}) is $(\Lambda^{\{-1,0,1\}}_1 \inc
\Delta^{\{-1,0,1\}})\jbox (\emptyset \inc K)$, hence is inner anodyne
by Lemma~\ref{le:joinbox}.  So there is a lifting $G$ in the above
diagram.
The adjoint $G^\sharp\taking\Delta^{\{-1,0,1\}}\to X_{/K}$ of $G$ is a right inverse provider for $f$.

Let $h$ denote the map $\Delta^{\{-1,0\}}\ra X_{/K}$ obtained by
restricting $G^\sharp$.  So $h$ is a right inverse for $f$.
The composite $h'\colon \Delta^{\{-1,0\}}\ra X_{/K}\ra X$ will be a right inverse for $f'$.  Hence $h'$ is a quasi-isomorphism and
itself has a right inverse.  Repeating exactly the same argument as
above, but replacing all occurences of $f$ and $f'$ by $h$ and $h'$,
we can construct a $k\colon \Delta^1\ra X_{/K}$ giving a right
inverse for $h$.  Returning to our outline from the first paragraph,
this completes the proof.
\end{proof}

\begin{prop}\label{prop:special left horn-absolute}
Let $X$ be a quasi-category and suppose given
a solid arrow diagram
$$\xymatrix{\Lambda^n_0\ar[r]^p\ar[d]& X\\
\Delta^n\ar@{.>}[ur]}
$$ 
in which $p$ is a special left horn.  Then
there exists a dotted lift as shown.  

A dotted lift also exists when $p$ is replaced by a special right horn $q\taking\Lambda^n_n\to X$.
\end{prop}

\begin{proof}
We only prove the first statement; the second follows from a dual argument.  To do so, we will produce a solid-arrow diagram
\begin{align}\label{dia:liftingspecialhorns-absolute}
\xymatrix{
\Lambda^{\{0,1,2,\ldots,n\}}_0\ar[r]\ar@/^1.6pc/[rr]^p\ar[d]
&\Lambda^{\{0,1,0',2,\ldots,n\}}_{0'}\ar[r]^<<{p'}\ar[d]
&X\\
\Delta^{\{0,1,2,\ldots,n\}}\ar[r]^{\delta^{\{0'\}}}
&\Delta^{\{0,1,0',2,\ldots,n\}}\ar@{.>}[ur]
}
\end{align}
and the dotted arrow (which will exist because $X$ is a quasi-category)
will compose with $\delta^{\{0'\}}$ to produce the desired lift.

To construct $p'$ we must extend $p\colon \Lambda^n_0\ra X$ to
$\Lambda^{\{010'2\ldots n\}}_{0'}$.
%\[ \xymatrix{
%\Lambda^n_0 \ar[r]^p \ar@{ >->}[d] & X  \\
%\Lambda^{\{010'2\ldots n\}}_{0'}. \ar@{.>}[ur]
%}
%\]
Write $Z=\Lambda^{\{010'2\ldots n\}}_{0'}$ and $Z_0=\Lambda^n_0$.  
Let $Z_1=Z_0\cup [00'23\ldots n]$, and let
$Z_2=Z_1\cup [\Delta^{\{010'\}} \join \partial \Delta^{\{23\ldots n\}}]$.  We now
have a filtration
\[ Z_0 \subseteq Z_1 \subseteq Z_2 \subseteq Z_3=Z,
\]
and we will construct $p'$ on each term. 

Extend $p$ to $Z_1$ via the composite $\Delta^{\{00'23\ldots n\}} \ra
\Delta^{\{023\ldots n\}} \ra X$ (in other words, via a degeneracy).  Then consider the following diagram
\[ \xymatrix{
[\Delta^{\{00'\}} \join \partial\Delta^{\{23\ldots n\}}]\cup
[\Delta^{\{01\}} \join \partial\Delta^{\{23\ldots n\}}] \ar[r]\ar[d] &
Z_1 \ar[r]^{p'}\ar[d] & X \\
[\Delta^{\{010'\}} \join \partial \Delta^{\{23\ldots n\}}] \ar[r] & Z_2
\ar@{.>}[ur]
}
\]
and note that the square is a pushout.
By adjointness we need to construct a lift in the diagram  
\[ \xymatrix{
\Delta^{\{00'\}} \cup \Delta^{\{01\}} \ar[r]\ar[d] &
X_{/\Delta^{\{23\ldots n\}}}  \\
 \Delta^{\{010'\}\ar@{.>}[ur]} 
}
\]
but such a lift exists by Proposition~\ref{pr:conservative} (applied to the quasi-isomorphism $p|_{\Delta^{\{01\}}}$).
After adjointing, this defines $p'$ on $Z_2$.

Finally, we note that $Z_2\inc Z_3$ is a cobase change of an inner
horn inclusion.  Indeed, the only nondegenerate simplex lying in
$Z_3-Z_2$ is $[10'23\ldots n]$, and all of its faces containing $0'$
lie in $Z_2$.  So $Z_2\inc Z_3$ is a cobase change of
$\Lambda^{\{10'2\ldots n\}}_{0'} \inc \Delta^{\{10'2\ldots n\}}$.  
Since $X$ is a quasi-category, we may extend our lift $p'$ from
$Z_2$ to $Z_3$.  This completes the construction of $p'$, and thereby
also completes the proof.

\end{proof}

We will also have need for a relative version of the above result:

\begin{prop}\label{prop:special left horn}
Let $X$ and $Y$ be quasi-categories.
Suppose given a solid arrow diagram
$$\xymatrix{\Lambda^n_0\ar[r]^p\ar[d]& X\ar[d]\\
\Delta^n\ar@{-->}[ur] \ar[r]_m & Y}
$$ 
in which $p$ is a special left horn and $X\ra Y$ is an inner fibration.  Then
there exists a dotted arrow making the diagram commute.

A dotted lift also exists when $p$ is replaced by a special right horn $q\taking\Lambda^n_n\to X$.

\end{prop}

Before proving this we need a lemma, which takes care of a special
case.  The proof of the lemma uses Proposition~\ref{prop:special left
horn-absolute}.

\begin{lemma}
\label{le:special case}
Let $X$ and $Y$ be quasi-categories, and suppose given a diagram
\[ \xymatrix{ \Lambda^2_0 \ar[r]^{p}\ar[d] & X \ar[d]^a \\
 \Delta^2 \ar[r]_m & Y
}
\]
where $a$ is an inner fibration, $p|_{[01]}$ is a quasi-isomorphism,
and $p|_{[02]}$ is degenerate.  Then
the above square has a lift.

A similar result holds when $p$ is a map $\Lambda^2_2\ra X$ and
$p|_{[12]}$ is a quasi-isomorphism.
\end{lemma}

\begin{proof}
We only prove the first statement, the second one being dual.
Let $f=p|_{[01]}$,
and let $\sigma\colon \Delta^{\{0,1,2\}}\ra X$ be a left inverse
provider for $f$.  Let $h=\sigma|_{[12]}$.

Now define a map $\tau\colon \Lambda^3_0 \ra Y$ by
$\tau|_{[012]}=a\circ \sigma$, $\tau|_{[013]}=m$,
and where $\tau|_{[023]}$ is a double degeneracy.  Then since
$\tau|_{[01]}$ is a quasi-isomorphism (being the image of $f$), it
follows by Proposition~\ref{prop:special left horn-absolute} that we
may extend $\tau$ to a map $\Delta^3\ra Y$; call this extension $\tau$
as well.

As $X\ra Y$ is an inner fibration, we may choose a lift in the
following square:
\[ \xymatrixcolsep{2.5pc}\xymatrix{
\Lambda^{\{1,2,3\}}_2 \ar[r]^{h\Wedge {*}}\ar[d] & X \ar[d] \\
\Delta^{\{1,2,3\}} \ar[r]_-{\tau|_{[123]}}\ar@{.>}[ur]^\lambda & Y. 
}
\]  
Extend $\lambda$ to a map $\Lambda^{\{0,1,2,3\}}_1\ra X$ by 
$\lambda|_{[012]}=\sigma$ and letting $\lambda|_{[023]}$ be a double degeneracy.
Then we have the diagram
\[ \xymatrix{
\Lambda^3_2 \ar[r]^{\lambda}\ar[d] & X \ar[d] \\
\Delta^3 \ar[r]_\tau\ar@{.>}[ur]^\omega & Y, 
}
\]  
and there must exist a lifting $\omega$.  One checks that
$p=\lambda|_{\Lambda^{013}_0}$ and $m=\tau|_{\Delta^{013}}$, so
$\omega|_{[013]}$ provides the lift for our original square.
\end{proof}

\begin{proof}[Proof of Proposition~\ref{prop:special left horn}]
We only prove the first statement; the second follows from a dual
argument.  The proof follows the same general outline as that of
Proposition~\ref{prop:special left horn-absolute}.  We will produce a
solid-arrow diagram
\begin{myequation}
\label{dia:liftingspecialhorns}
\xymatrix{\Lambda^{\{0,1,2,\ldots,n\}}_0\ar[r]\ar@/^1.6pc/[rr]^p\ar[d]&\Lambda^{\{0,1,0',2,\ldots,n\}}_{0'}\ar[r]^<<{p'}\ar[d]&X\ar[d]^h\\\Delta^{\{0,1,2,\ldots,n\}}\ar[r]^{\delta^{\{0'\}}}\ar@/_1.6pc/[rr]_m&\Delta^{\{0,1,0',2,\ldots,n\}}\ar@{.>}[ur]\ar[r]_<<{m'}&Y}
\end{myequation}
and the dotted arrow (which exists because $X\ra Y$ is an inner fibration)
will compose with $\delta^{\{0'\}}$ to produce the desired lift.

Our map $m\taking\Delta^n\ra Y$ may be regarded as a map $\Delta^{\{01\}}\join
\Delta^{\{23\ldots n\}} \ra Y$.  The adjoint $f\colon \Delta^{\{01\}}\ra
Y_{/\Delta^{\{2\ldots n\}}}$ is such that its composite with
$Y_{/\Delta^{\{2\ldots n\}}} \ra Y$ is a quasi-isomorphism, hence
by Proposition~\ref{pr:conservative} $f$ is itself a quasi-isomorphism.  So there is a
left inverse provider $\sigma\colon \Delta^{\{010'\}}\ra
Y_{/\Delta^{\{2\ldots n\}}}$ for $f$.   Let $m'$ be the adjoint of this
map.

To construct $p'$ we must produce a lift for the diagram
\[ \xymatrix{
\Lambda^{\{012\ldots n\}}_0 \ar[r]^-p \ar@{ >->}[d] & X \ar[d] \\
\Lambda^{\{010'2\ldots n\}}_{0'} \ar[r] & Y
}
\]
where the bottom horizontal map is the restriction of $m'$.  Define
$Z_0, Z_1, Z_2$ and $Z_3$ exactly as in the proof of Proposition~\ref{prop:special
left horn-absolute}; we will construct the lift $p'$ on each $Z_i$.  

Extend $p$ to $Z_1$ via the composite $\Delta^{\{00'23\ldots n\}} \ra
\Delta^{\{023\ldots n\}} \ra X$.  Continuing to argue as in the proof
of Proposition~\ref{prop:special left horn-absolute},
extending $p'$ to $Z_2$ amounts to constructing a lift in the diagram
\[ \xymatrix{
\Delta^{\{00'\}} \cup \Delta^{\{01\}} \ar[r]\ar[d] &
X_{/\Delta^{\{23\ldots n\}}} \ar[d] \\
 \Delta^{\{010'\}} \ar[r] & Y_{/\Delta^{\{23\ldots n\}}}.
}
\]
But this has a lift by Lemma~\ref{le:special case}.
After adjointing, this defines $p'$ on $Z_2$.

Finally, the extension of $p'$ from $Z_2$ to $Z_3$ follows exactly as
in Proposition~\ref{prop:special left horn-absolute}.  
This completes the construction of $p'$, and
also completes the proof.
\end{proof}

\subsection{Consequences of special outer horn lifting}

\begin{proof}[Proof of Proposition~\ref{pr:quasi2}]
This is precisely Proposition~\ref{prop:special left horn-absolute}.
\end{proof}

\begin{prop}
\label{pr:special-fib-matching}
Let $f\colon X\ra Y$ be an inner fibration between quasi-categories
and suppose that $f$ has the
RLP with respect to $\{0\}\inc E^1$.  Then $f$ also has the RLP with
respect to
\[ (\{0\}\inc E^1)\bbox (A\inc B) \]
for any monomorphism $A\inc B$.  Equivalently, the map $X^{E^1}\ra
X^{\{0\}}\times_{Y^{\{0\}}} Y^{E^1}$ is a Kan acyclic fibration.
\end{prop}

\begin{proof}
By Lemma~\ref{le:box-special} and
Proposition~\ref{prop:special left horn} we know that $f$ has the RLP
with respect to $(\{0\}\inc E^1)\bbox(\bdd{r}\inc \del{r})$ for any
$r>0$.  The assumptions on $f$ take care of the case $r=0$.  Since
every monomorphism is generated by the boundary inclusions
$\bdd{r}\inc \del{r}$, the result follows by an easy induction on
simplices.
\end{proof}

\begin{proof}[Proof of  Proposition~\ref{pr:quasi1}(a)]
This follows immediately from
Proposition~\ref{pr:special-fib-matching}, since $X\ra *$ has the RLP
with respect to $\{0\}\inc E^1$.  
\end{proof}

Our last task is to prove Proposition~\ref{pr:quasi0}.  This is 
taken care of by the following lemma, which we learned from 
Nichols-Barrer \cite{N}.

\begin{lemma}
\label{le:NB}
Let $h\taking X\to Y$ be an inner fibration and suppose that
$f\taking\Delta^1\to X$ is a quasi-isomorphism.
Then for every solid-arrow commutative diagram of
the form
$$\xymatrix{\Delta^1\ar[r]^f\ar[d]&X\ar[d]^h\\E^1\ar[r]\ar@{-->}[ur]^\tau&Y,}$$
there exists a dotted arrow making the diagram commute.
\end{lemma}

\begin{proof}
Each simplex in $E^1$ is determined by its sequence of vertices.  
We can therefore denote any $a\in E^1_n$ by a sequence
$[a_0,a_1,\ldots,a_n]$ where each $a_i\in\{0,1\}$.  
Define a filtration $Z_0\subseteq Z_1\subseteq \cdots \subseteq E^1$
by letting $Z_0=[0]$, $Z_1=[01]$, $Z_2=[010]$, and so on.  The subset
$Z_n$ is the unique nondegenerate $n$-simplex whose vertex sequence starts with $0$.  Note that $E^1=\bigcup_{n\in\N}Z_n$.

We claim that for each $n\geq 1$ there is a pushout diagram
\[ \xymatrix{ \Lambda^{n+1}_0 \ar[r]^g\ar@{ >->}[d] & Z_n \ar[d] \\
      \Delta^{n+1} \ar[r] & Z_{n+1}.
}
\]
Here the map $g$ sends each simplex $[01\cdots \hat{i}\cdots n+1]$, for
$0<i\leq n+1$, to the simplex in $E^1$ specified by reducing all the
entries in the vertex sequence modulo $2$.  It is easy to check that
this description is compatible on the overlap between simplices,
therefore defines a map  $\Lambda^{n+1}_0\ra Z_n$, and that the
pushout is $Z_{n+1}$.  

Starting with a lifting square as in the statement of the lemma, one
inductively produces lifts $Z_n\ra X$ using special outer horn lifting
(Proposition~\ref{prop:special left horn}).  
Taking the colimit gives the resulting lift $E^1\ra Y$.
\end{proof}

\begin{proof}[Proof of Proposition~\ref{pr:quasi0}]
We have already proved (ii)\!\!$\iff$\!\!(iii) as part of
Proposition~\ref{pr:3iso-again}; and (i)$\Rightarrow$(ii) is obvious.
Finally, (ii)$\Rightarrow$(i) results
immediately from applying Lemma \ref{le:NB} to $X\ra *$.
\end{proof}

At this point we have proven all the assertions in
Section~\ref{se:quasi} except for the existence of the Joyal model category
structure.  We take up that in the next section.

%%%%%%%%%%%%%%%%%%%%%%%%%%%%%%%%%%%%%%%%%%%%%%%%%%%%%%%%%%%%%

%%%%%%%%%%%%%%%%%%%%%%%%%%%%%%%%%%%%%%%%%%%%%%%%%%%%%%%%%%%%%

\section{Development of the Joyal model structure}
\label{se:model}
In this section we prove the existence of the Joyal model structure.
The approach
we follow is entirely due to Joyal \cite{J2}, our only contribution
being to streamline the presentation so that it occupies only a few
pages.

\medskip

Say that a map $X\ra Y$ is a \dfn{special inner fibration} if it has the
right-lifting-property with respect to all inner horn inclusions as
well as with respect to the map $\{0\}\inc E^1$.  Note that such a map
also has the RLP with respect to $\{1\}\inc E^1$, using the evident
automorphism of $E^1$.   Also note that if $X$ is a quasi-category
then $X\ra *$ is a special inner fibration (using the retraction
$E^1\ra \{0\}$).  

\begin{lemma}
\label{le:special-inner1}
Let $X$ and $Y$ be quasi-categories.  
If $X\ra Y$ is a special inner fibration and $A\ra B$ is a
monomorphism, then $X^B\ra X^A\times_{Y^A} Y^B$ is also a special inner
fibration.
\end{lemma}

\begin{proof}
We know by Proposition~\ref{pr:box-inner} that $X^B\ra X^A\times_{Y^A}
Y^B$ is an inner fibration.  Using
Proposition~\ref{pr:special-fib-matching} we also know it has the RLP
with respect to $\{0\}\inc E^1$.  
\end{proof}

\begin{lemma}
\label{le:special-inner2}
Let $f \colon X\ra Y$ be a special inner fibration between
quasi-categories, and assume that $f$ is also a Joyal equivalence.
Then $f$ has the RLP with respect to both $\emptyset \ra \Delta^0$ and
$\{0,1\}\inc E^1$.
\end{lemma}

\begin{proof}
%The RLP with respect to $\emptyset\ra \Delta^0$ is easy.  First note
%that $f$ is an $E^1$-homotopy equivalence, by
%Remark~\ref{re:E1-homotopy}.  Therefore if $y$ is a
%$0$-simplex in $Y$, there exists a $0$-simplex $x\in X$ and a map
%$E^1\ra Y$ sending $0\mapsto f(x)$ and $1\mapsto y$.  Using that $f$
%has the RLP with respect to $\{0\}\inc E^1$, we get a map $E^1\ra X$
%where the image of $1$ lifts the original $0$-simplex $y$.

Recall from Section~\ref{se:quasi} the subcomplexes $J(X)\subseteq X$ and $J(Y)\subseteq Y$
consisting of all simplices whose $1$-faces are quasi-isomorphisms.
Note in particular that $X$ and $J(X)$ have the same $0$-simplices.

It is easy to see that since $X\ra Y$ is a special inner fibration,
$J(X)\ra J(Y)$ is a Kan fibration (lifting with respect to inner horns
is easy, and for outer horns follows from
Proposition~\ref{prop:special left horn}).
Moreover, $X\ra Y$ is an
$E^1$-homotopy equivalence by Remark~\ref{re:E1-homotopy}.  Since
$J(X\times E^1)\iso J(X)\times J(E^1)=J(X)\times E^1$, it follows that
$J(X)\ra J(Y)$ is an $E^1$-homotopy equivalence as well.  This implies
that it is a Kan equivalence, therefore it is a Kan acyclic
fibration.  

But now observe that any lifting square
\[ \xymatrix{
\{0,1\} \ar[r]\ar@{ >->}[d] & X \ar[d]^f \\
E^1 \ar[r]_h & Y
}
\]
necessarily factors through $J(X)\ra J(Y)$, and therefore has a lift.
The same is true for lifting squares with $\emptyset\ra \Delta^0$.  
This completes the proof.
\end{proof}

We need one more lemma:

\begin{lemma}
\label{le:model-key}
Let $X$ and $Y$ be quasi-categories.  
\begin{enumerate}[(a)]
\item If $X\ra Y$ is a special inner fibration and a Joyal equivalence, then
it is a Kan acyclic fibration.
\item If $A\ra B$ is a monomorphism and a Joyal equivalence, then
$X^B\ra X^A$ is a Kan acyclic fibration.
\end{enumerate}
\end{lemma}

\begin{proof}
For (a), we will show that $X\ra Y$ has the RLP with respect to any
monomorphism $A\inc B$.  This is equivalent to proving that $X^B\ra
[X^A\times_{Y^A} Y^B]$ is surjective on $0$-simplices.  Consider the diagram
\[ \xymatrix{
X^B\ar[dr]\ar[drr]\ar[ddr] \\
& P \ar[r]\ar[d] & Y^B \ar[d] \\
& X^A \ar[r] & Y^A,
}
\] 
where $P$ is the pullback,
and suppose $w$ is a $0$-simplex of $P$.  Write $w_1$, $w_2$, and
$w_3$ for the images of $w$ in $Y^B$, $Y^A$, and $X^A$, respectively.

By Remark~\ref{re:E1-homotopy}, $X\ra Y$ is an $E^1$-homotopy equivalence, from
which it follows at once that $X^B\ra Y^B$ is also an $E^1$-homotopy
equivalence.  Hence $X^B\ra Y^B$ is a Joyal equivalence, and it is a
special inner fibration by Lemma~\ref{le:special-inner1}.
By Lemma~\ref{le:special-inner2} it is therefore
surjective, so there is a $0$-simplex $x\in X^B$ whose image is
$w_1$.  Let $x_3$ denote the image of $x$ in $X^A$.  Then $x_3$ and
$w_3$ have the same image in $Y^A$, so we have a square
\[ \xymatrix{
\{0,1\} \ar[r]^{x_3\amalg w_3}\ar@{ >->}[d] & X^A \ar[d] \\
E^1 \ar[r]_{w_2} & Y^A
}
\]
where the bottom map collapses everything to $w_2$.  But
$X^A\ra Y^A$ is a special inner fibration and a Joyal
equivalence (by the same arguments used for $X^B\ra Y^B$), so
Lemma~\ref{le:special-inner2} shows that the above square has a
lifting $\lambda\colon E^1\ra X^A$.  Consider the induced map $\tilde{\lambda}\taking E^1\ra
P$ which projects to $w_1$ in $Y^B$, $w_2$ in $Y^A$, and $\lambda$ in
$X^A$.  

At this point we have a new square
\[ \xymatrix{
\{0\} \ar[r]^{x}\ar@{ >->}[d] & X^B \ar[d] \\
E^1 \ar[r]_{\tilde{\lambda}} & P,
}
\]
which has a lift because Lemma~\ref{le:special-inner1} tells us that 
$X^B\ra P$ is a special inner fibration.  The
image of the vertex $1$ is the desired preimage of the original $0$-simplex $w$.

To prove (b), first note that $X^B\ra X^A$ is a special inner fibration
between quasi-categories.  Using (a), it suffices to show that this map
is also a Joyal equivalence.  Let $S$ be any simplicial set.  Then
$[B,X^S]_{E^1} \ra [A,X^S]_{E^1}$ is a bijection for every
quasi-category $X$.  But $[B,X^S]_{E^1}=[B\times
S,X]_{E_1}=[S,X^B]_{E^1}$, and similarly with $B$ replaced by $A$.
Hence
\[ [S,X^B]_{E^1} \ra [S,X^A]_{E^1} \]
is a bijection for every simplicial set $S$.  Taking $S=X^A$ and
$S=X^B$ one sees at once that $X^B\ra X^A$ is an $E^1$-homotopy
equivalence, hence it is a Joyal equivalence by Remark~\ref{re:E1-homotopy}.
\end{proof}

We now prove the existence of the Joyal model structure:

\begin{proof}[Proof of Theorem~\ref{th:model}]
We use a result of Jeff Smith's about producing model structures on
locally presentable categories, written up by Beke.  By
\cite[Theorem 1.7 and  Proposition 1.15]{Be} we are guaranteed the existence of a cofibrantly-generated model
structure on $\sSet$ where the cofibrations are the monomorphisms and
the weak equivalences are the Joyal equivalences if we can verify the
following:
\begin{enumerate}[(1)]
\item The Joyal equivalences are closed under retracts and satisfy the
two-out-of-three property;
\item Every Kan acyclic fibration is a Joyal equivalence;
\item The class of cofibrations which are Joyal equivalences is closed
under pushouts and transfinite composition;
\item The class of Joyal equivalences is an accessible class of maps,
in the sense of \cite[Definition 1.14]{Be}.
\end{enumerate}

Point (1) is obvious.   For (2), if $f\colon X\ra Y$ is a Kan acyclic
fibration then there is a map $\chi\colon Y\ra X$ such that
$f\chi=\id$.  Then the two maps $f\chi f$ and $f$ are equal, hence
they are $E^1$-homotopic (trivially).  So we have a square
\[ \xymatrix@=.37in{
X\amalg X \ar[r]^-{\id \amalg (\chi f)}\ar@{ >->}[d] & X  \ar@{->>}[d]^\sim \\
X\times E^1 \ar[r] & Y
}
\]
and this square must have a lifting.  This shows that $f$ is an
$E^1$-homotopy equivalence, hence it is a Joyal equivalence by
Remark~\ref{re:E1-homotopy}.  

We prove (3) by showing that a monomorphism $f\colon A\inc B$ is a
Joyal equivalence if and only if it has the left-lifting-property with
respect to special inner fibrations between quasi-categories.  Since
monomorphisms and maps with a left-lifting-property are closed under
pushouts and transfinite compositions, this will be enough.

First suppose that $f$ is a monomorphism with the indicated
left-lifting-property.  Then, in particular, it has this property with
respect to the maps $X\ra *$ and $X^{E^1}\ra X^{\{0,1\}}$ for every
quasi-category $X$ (using Lemma~\ref{le:special-inner1}).  It follows
immediately that $f$ is a Joyal equivalence.

Now suppose that $f$ is a monomorphism and a Joyal equivalence.  For
any special inner fibration $X\ra Y$ between quasi-categories, consider the
diagram
\[ \xymatrix{
X^B\ar[dr]\ar[drr]\ar[ddr] \\
& P \ar[r]\ar[d] & Y^B \ar[d] \\
& X^A \ar[r] & Y^A,
}
\]
where $P=X^A\times_{Y^A} Y^B$.  The map $X^B\ra P$ is a special inner
fibration by Lemma~\ref{le:special-inner1}.    
 By
Lemma~\ref{le:model-key}(b), the maps $Y^B\ra Y^A$ and $X^B\ra X^A$
are Kan acyclic fibrations. 
Hence the pullback $P\ra X^A$ is also a
Kan acyclic fibration.  But then $P\ra X^A$ and $X^B\ra X^A$ are both
Joyal equivalences, hence so is $X^B\ra P$.  By 
Lemma~\ref{le:model-key}(a)  we have that $X^B\ra P$ is a trival Kan
fibration, hence it is surjective.  But this precisely says that $A\ra
B$ has the left-lifting-property with respect to $X\ra Y$.
This now completes the proof of (3).

Finally, for point (4) we argue as follows.  Let $S$ be the set
consisting of all inner horn inclusions together with the map
$\{0\}\inc E^1$. Consider the functorial factorizations
\[ [X\llra{f} Y] \mapsto [X \llra{i_f} P_f \llra{p_f} Y]
\]
provided by the small object argument.  Here
 $X\ra P_f$ is a transfinite composition of pushouts of
coproducts of the
maps in $S$, and $P_f\ra Y$ has the right-lifting-property with
respect to maps in $S$.  It is an observation of Smith's (and a
straightforward exercise) that these
factorization functors preserve $\lambda$-filtered colimits for large
enough regular cardinals $\lambda$.  Also note that $X\ra P_f$ is a
Joyal equivalence, using point (3) together with the fact that the
maps in $S$ are Joyal equivalences and cofibrations.

For brevity write $X\to L_S(X)$ for the factorization applied to
$X\ra *$.  For $f\colon X\ra Y$, let $R(f)$ be the map indicated in
the following diagram:
\[ \xymatrix{ X \ar[rr]^f\ar@{ >->}[d]^\sim && Y\ar@{ >->}[d]^{\sim} \\
 L_S(X) \ar[dr]_{\sim}\ar[rr]^{L_Sf} && L_S(Y) \\
  &P_{L_Sf}.\ar[ur]_{R(f)=p_{L_S(f)}}
}
\]
Note that $R(f)$ is a special fibration between quasi-categories, and that  
the indicated maps are Joyal equivalences.
It
is easy to see that $f$ is a Joyal equivalence if and only if $R(f)$
is a Joyal equivalence; by Lemma~\ref{le:model-key}(a) together
with point (2), the latter is true if and only
if $R(f)$ is a Kan acyclic fibration.  So the class of Joyal
equivalences is $R^{-1}(\cT)$, where $\cT$ is the class of Kan acyclic
fibrations.  But $R$ is an accessible functor because it preserves
large enough filtered colimits, and $\cT$ is an accessible class by
Lemma~\ref{le:accessible} below. So
by \cite[Prop. 1.18]{Be} 
the class of Joyal equivalences is also accessible.

This completes the proof of the existence of a model structure where
the cofibrations are the monomorphisms and the weak equivalences are
the Joyal equivalences.  To characterize the fibrant objects, note
first that it follows at once from
Proposition~\ref{pr:first-properties}(c) 
that every fibrant object will be a quasi-category.
Let $X$ be a quasi-category, and let $X\cof
\hat{X}$ be a fibrant-replacement in our model structure.  By
Proposition~\ref{pr:first-properties}(b), there is a lifting
\[ \xymatrix{
X \ar@{=}[r] \ar@{ >->}[d] & X \\
\hat{X} \ar@{.>}[ur]
}
\]
But then $X$ is a retract of the fibrant object $\hat{X}$, hence $X$
is fibrant.

For uniqueness of the model structure, see the explanation in
\cite[Section 2.3]{DS1}.
\end{proof}

The following lemma was used in the above proof:

\begin{lemma}
\label{le:accessible}
The class of Kan acyclic fibrations in $\sSet$ is an accessible class
of maps in the sense of \cite[Definition 1.14]{Be}.
\end{lemma}

\begin{proof}
It is easy to check that the category of surjections in $\Set$ is an
accessible class.  Consider the functor $G\colon \Mor(\sSet) \ra
\Mor(\Set)$ which sends a map $f\colon X\ra Y$ to the map
\[  \coprod_{n\geq 0} \, \bigl [X^{\del{n}}  \ra
X^{\bdd{n}}\times_{Y^{\bdd{n}}} Y^{\del{n}} \bigr ]
\]
restricted to the $0$-simplices.  Note that $G$ preserves filtered
colimits, because the functors $(\blank)^{\del{n}}$ and
$(\blank)^{\bdd{n}}$ do, and hence $G$ is an accessible functor.  The
class of Kan acyclic fibrations is exactly the inverse image under $G$ of
the surjections; hence by
\cite[Proposition 1.18]{Be} the class of Kan acyclic fibrations is an
accessible class of maps.
\end{proof}

%%%%%%%%%%%%%%%%%%%%%%%%%%%%%%%%%%%%%%%%%%%%%%%%%%%%%%%%%%%%%%%%%%%%%%%%%%%%%%%%%%%%%%%

\bibliographystyle{amsalpha}

\end{document}